\documentclass[]{class2}

\usepackage{graphicx, xfrac, lineno, float, subcaption, enumitem, xcolor, booktabs, multirow}
\usepackage[normalem]{ulem}

\usepackage{csquotes}
\renewcommand{\mkbegdispquote}[2]{\itshape}

\usepackage[symbol]{footmisc}

\begin{document}


\begin{frontmatter}

\titledata{Some snarks are worse than others}{}

\authordata{Edita M\'{a}\v{c}ajov\'{a}}
{Department of Computer Science, Comenius University, 842 48 Bratislava, Slovakia}
{macajova@dcs.fmph.uniba.sk}
{}

\authordata{Giuseppe Mazzuoccolo}
{Dipartimento di Informatica, Universit\`{a} di Verona, Strada Le Grazie 15, Verona, Italy}
{giuseppe.mazzuoccolo@univr.it}
{}

\authordata{Vahan Mkrtchyan}
{Dipartimento di Informatica, Universit\`{a} di Verona, Strada Le Grazie 15, Verona, Italy\\
Gran Sasso Science Institute, L'Aquila, Italy}
{vahan.mkrtchyan@gssi.it}
{}

\authordata{Jean Paul Zerafa}
{Dipartimento di Scienze Fisiche, Informatiche e Matematiche,\\ Universit\`{a} di Modena e Reggio Emilia, Via Campi 213/B, Modena, Italy}
{jeanpaul.zerafa@unimore.it}
{}


\begin{abstract}
Many conjectures and open problems in graph theory can either be reduced to cubic graphs or are directly stated for cubic graphs. Furthermore, it is known that for a lot of problems, a counterexample must be a snark, i.e. a bridgeless cubic graph which is not 3--edge-colourable. In this paper we deal with the fact that the family of potential counterexamples to many interesting conjectures can be narrowed even further to the family ${\cal S}_{\geq 5}$ of bridgeless cubic graphs whose edge set cannot be covered with four perfect matchings. The Cycle Double Cover Conjecture, the Shortest Cycle Cover Conjecture and the Fan-Raspaud Conjecture are examples of statements for which ${\cal S}_{\geq 5}$ is crucial.

In this paper, we study parameters which have the potential to further refine ${\cal S}_{\geq 5}$ and thus enlarge the set of cubic graphs for which the mentioned conjectures can be verified. We show that ${\cal S}_{\geq 5}$ can be naturally decomposed into subsets with increasing complexity, thereby producing a natural scale for proving these conjectures. More precisely, we consider the following parameters and questions: given a bridgeless cubic graph, (i) how many perfect matchings need to be added, (ii) how many copies of the same perfect matching need to be added, and (iii) how many 2--factors need to be added so that the resulting regular graph is Class I? 
We present new results for these parameters and we also establish some strong relations between these problems and some long-standing conjectures.
\end{abstract}

\end{frontmatter}

\section{Motivation}

It is well-known that many long standing conjectures in graph theory can be reduced to the class of cubic graphs. That is, if one can prove such a conjecture for all cubic graphs then the general statement for arbitrary graphs will immediately follow. The Cycle Double Cover Conjecture \cite{ZhangBook} and the $5$--Flow Conjecture \cite{Tutte5FC} fall into this category. Some other well-known conjectures are formulated directly for cubic graphs such as the Petersen Colouring Conjecture \cite{JaegerPCC} and the Berge-Fulkerson Conjecture \cite{BergeFulk}.

In all mentioned conjectures, only a very small subset of all cubic graphs is critical for proving them. A classical result by Vizing \cite{Vizing} naturally divides cubic graphs in two classes. More precisely, Vizing's Theorem divides simple graphs in two classes according to the value of the chromatic index $\chi'$ with respect to the maximum degree $\Delta$. A simple graph $G$ has $\chi'(G)$ either equal to $\Delta(G)$ or to $\Delta(G)+1$, and is said to be a \emph{Class I} or \emph{Class II} graph, respectively. In case of multigraphs $G$, we say that $G$ is \emph{Class I} if $\chi'(G)=\Delta(G)$, and \emph{Class II}, otherwise.

The bridgeless cubic graphs having chromatic index $4$ will be referred to as \emph{snarks}, and we will denote the set of all snarks by $\cal S$. It can be easily shown that the study of all previously mentioned conjectures can be restricted to $\cal S$, since all such conjectures are true for $3$--edge-colourable cubic graphs. We remark that in literature one may find a stronger and more refined definition of snarks, which refers only to those graphs in $\cal S$ which are cyclically 4--edge-connected and with girth at least 5. In this paper we shall use the broader definition of snarks and shall refer to the more specifically defined snarks as \emph{non-trivial snarks}.

In addition, the class of snarks relevant for some old and new problems can be further restricted to a specific subset of $\cal S$, which we shall denote by ${\cal S}_{\geq 5}$ (see definition below). More precisely, ${\cal S}_{\geq 5}$ is shown to be critical for several, seemingly unrelated, problems. In order to define the class ${\cal S}_{\geq 5}$, we need the following parameter:

\begin{definition}
The \emph{perfect matching index} of a graph $G$, denoted by $\chi_{e}'(G)$, is the minimum number of perfect matchings of $G$ whose union covers the whole set $E(G)$. If such a number does not exist, $\chi_{e}'(G)$ is defined to be infinity. This parameter is also known in literature as the \emph{excessive index} of a graph (see \cite{BonCar}).
\end{definition}

Since bridges in cubic graphs belong to every perfect matching, the perfect matching index of a cubic graph having a bridge is infinite. Consequently, in what follows, we shall only consider bridgeless cubic graphs. Trivially, the chromatic index $\chi'(G)$ of a cubic graph $G$ is $3$ if and only if its perfect matching index is $3$, and so, the two parameters coincide for Class I cubic graphs. The same cannot be said for Class II bridgeless cubic graphs. Indeed, there exist examples of such graphs having perfect matching index $4$ and others having perfect matching index $5$, such as the well-known Petersen graph. In what follows, we denote the set of snarks having perfect matching index equal to $4$ by ${\cal S}_4$ and the set of snarks having perfect matching index at least $5$ by ${\cal S}_{\geq 5}$. 
Consequently, the following holds: 
\[{\cal S}= {\cal S}_4 \cup {\cal S}_{\geq 5} .\]
The above situation is summarised in Table \ref{TableChromExc}. Clearly, the Berge-Fulkerson Conjecture implies that all bridgeless cubic graphs have perfect matching index at most 5, and in turn, the latter statement is equivalent to the conjecture attributed to Berge. These two conjectures were in fact shown to be equivalent by the second author in \cite{MazzuoccoloEquivalence}, and we shall refer interchangeably to each of them as the Berge-Fulkerson Conjecture. If this conjecture is shown to be true, it would imply that all snarks in $ {\cal S}_{\geq 5}$ have perfect matching index exactly equal to $5$.

\begin{table}
\centering
\begin{tabular}{ c   c  c }
\toprule 
  Cubic graph $G$ & $\chi'(G)$ & $\chi'_e(G)$ \\

 \midrule
CLASS I & $3$ & $3$  \\
\midrule
CLASS II (${\cal S}$) & $4$ & \begin{tabular}{cl}
         $4$ & $({\cal S}_4)$ \\
         \midrule
         $\geq 5$ & $({\cal S}_{\geq 5})$ \\
        \end{tabular}
\\
\bottomrule
\end{tabular}
\caption{The relation between $\chi'(G)$ and $\chi'_{e}(G)$}
\label{TableChromExc}
\end{table}

The reason why the class ${\cal S}_{\geq 5}$ deserves particular attention not only in relation to the Berge-Fulkerson Conjecture but also with respect to other problems, is already very present in literature.
Moreover, from among more than sixty million non-trivial snarks of order at most 36 (see \cite{BrGHM}), only two  belong to ${\cal S}_{\geq 5}$, and both of them have perfect matching index equal to 5. This suggests that the subset of snarks that is substantial for many open problems is negligible compared to its complement. On the other hand, infinite classes of non-trivial snarks from ${\cal S}_{\geq5}$ are constructed in \cite{AbreuEtAl, WindmillEspMazz, EMMS}.

One of the most relevant results that shows the importance of the class ${\cal S}_{\geq 5}$ was
proven independently by Steffen \cite{Steffen} and by Hou et al. \cite{HouLaiZhang}, and states that each snark in ${\cal S}_4$ admits a cycle double cover. Thus, if a cubic graph is a counterexample to the Cycle Double Cover Conjecture, then it must belong to ${\cal S}_{\geq 5}$.

The Fan-Raspaud Conjecture \cite{FR} asserts that every bridgeless cubic graph $G$ admits three perfect matchings such that no edge of $G$ belongs to all three of them. This conjecture is obviously true for 3--edge-colourable cubic graphs and graphs from ${\cal S}_4$, making the family ${\cal S}_{\geq 5}$ critical once again.

Another unexpected relation seems to appear with the $5$--flow Conjecture. It is pointed out in \cite{AbreuEtAl} that all known examples of snarks with perfect matching index equal to $5$ also have circular flow number $5$ (see \cite{GoddynTarsiZhang} for a definition). In other words, it seems that all snarks having the largest possible perfect matching index according to the Berge-Fulkerson Conjecture, also have the largest possible circular flow number according to the $5$--Flow conjecture. We remark that the converse of the latter is false: there exists a large number of non-trivial snarks having circular flow number $5$ and perfect matching index $4$ (see \cite{GoedMattMazz}).

Let us mention a last example: the problem of finding a shortest cycle cover of a bridgeless graph (not necessarily cubic). A family~$\mathcal{C}$ of cycles of a graph $G$ is a \emph{cycle cover} of $G$ if every edge of $G$ is contained in at least one of the cycles in $\mathcal{C}$. The (total) length of a cycle cover $\mathcal{C}$ is the sum of the lengths of all the circuits making up the cycles in $\mathcal{C}$.

\begin{definition} \label{def:scc} Let $G$ be a bridgeless graph. The minimum total length over all possible cycle covers of $G$ is denoted by $scc(G)$, and a cycle cover having length $scc(G)$ is called a \emph{shortest cycle cover}.
\end{definition}

The Shortest Cycle Cover Conjecture by Alon and Tarsi \cite{AT} asserts that $scc(G)\le\sfrac{7}{5}\cdot |E(G)|$. In \cite{Steffen}, it is shown that if a graph $G$ belongs to ${\cal S}_{4}$, then $scc(G)=\sfrac{4}{3}\cdot|E(G)|$, thus leaving, once again, the conjecture open only for graphs from ${\cal S}_{\geq 5}$.\\

All previous example give a strong motivation to the study of the class  ${\cal S}_{\geq 5}$. 
In this paper, we study parameters which have a potential to further refine ${\cal S}_{\geq 5}$ and thus enlarge the set of cubic graphs for which the Cycle Double Cover Conjecture, the Fan-Raspaud Conjecture and other related problems can be proven. 
As a by-product, we also consider a parameter which identifies graphs in ${\cal S}_{4}$ that are, in a sense (explained later), closer to being 3--edge-colourable. Now we describe these parameters in more detail.

Let $G$ be any graph, and let $N\subseteq E(G)$. We denote by $G+N$ the multigraph obtained from $G$ after adding a parallel edge to every edge in $N$. In general, let $N_{1}, N_{2},\ldots, N_{t}\subseteq E(G)$. We denote by $G+N_{1}+\dots+N_{t}$, or equivalently by $G+\sum_{i=1}^{t}N_{i}$, the multigraph obtained by adding to every edge of $G$ a number of parallel edges equal to the number of times the original edge appears in $N_{1},\ldots, N_{t}$. In the special case when we add $t$ times the same set of edges $N$, the resulting graph is denoted by $G+tN$ (examples are given in Figure \ref{examplenotation}).

\begin{figure}[h]
      \centering
      \includegraphics[width=0.5\textwidth]{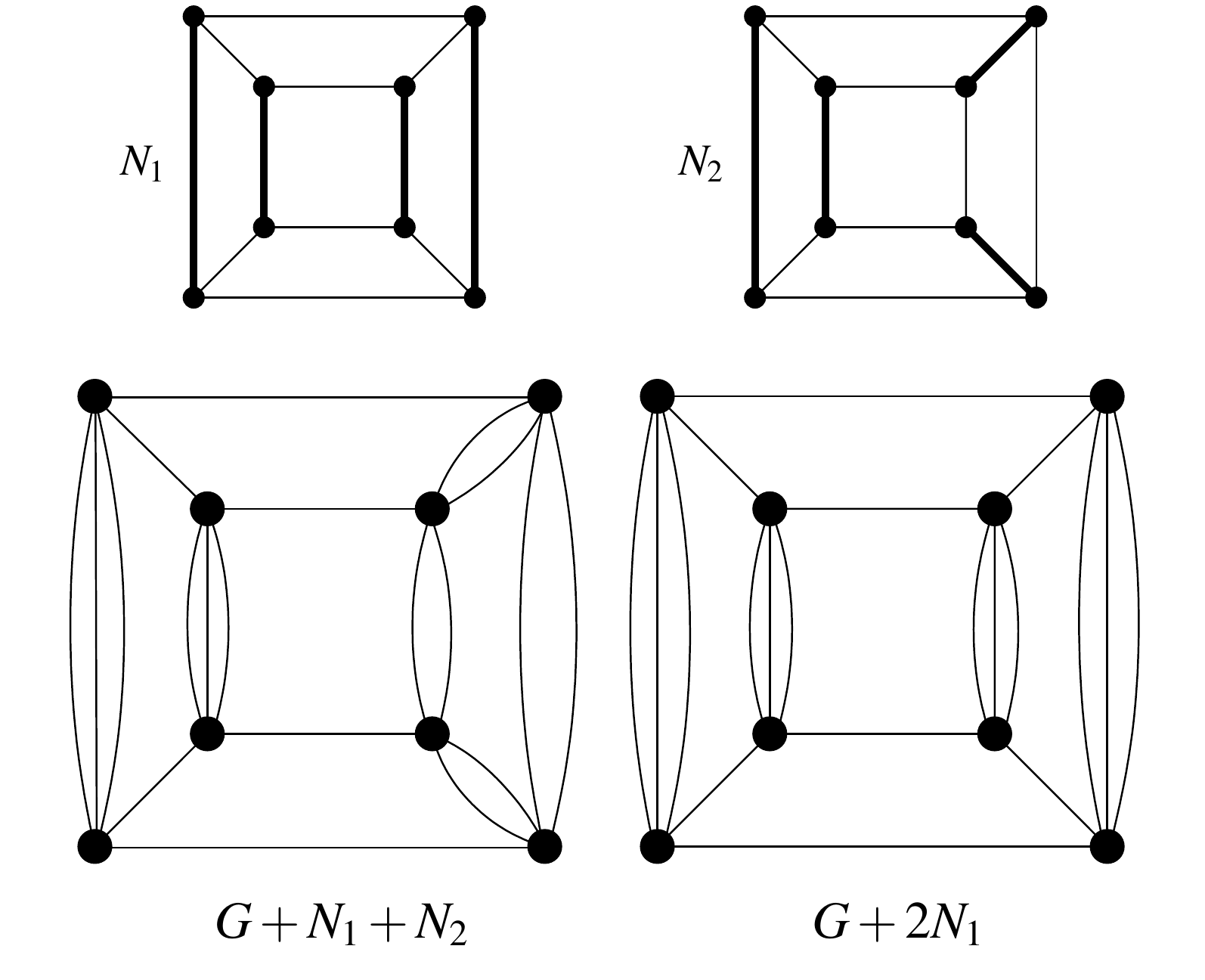}
      \caption{Perfect matchings $N_{1}$ and $N_{2}$ in $G$ and the graphs $G+N_{1}+N_{2}$ and $G+2N_{1}$}
      \label{examplenotation}
\end{figure}

The study of the following problem was firstly proposed to some of the authors by G. Brinkmann and E. Steffen during the workshop KOLKOM 2017 in Paderborn: given a bridgeless cubic graph $G$, when does there exist $k$ perfect matchings $M_1,...,M_k$ of $G$, for some integer $k\geq 0$ , such that the graph $G+M_1+...+M_k$ is $(k+3)$--edge-colourable or, equivalently, is Class I?

In the sequel, a $(k+3)$--edge-colouring of the multigraph $G+M_1+\ldots+M_k$ shall be sometimes considered as the proper edge-colouring of $G$ in which every edge $e$ is assigned $\nu(e)+1$ colours, where $\nu(e)$ is the number of times $e$ appears in the list $M_1,\ldots, M_k$. 

For a graph $G$, we define the following parameter related to this problem:
\begin{definition}
Denote by $l(G)$ the minimum number of perfect matchings needed to be added to $G$ such that the resulting graph is Class I. If such a number does not exist, then we set $l(G)=+\infty$.
\end{definition}

Obviously, for a cubic graph $G$, $l(G)=0$ if and only if $G$ is 3--edge-colourable. Observe also that the Berge Conjecture is true for cubic graphs $G$ with $l(G)\le2$. A slight variation of the previous definition will also be of interest later on in the paper.

\begin{definition}
Let $G$ be a graph admitting a perfect matching $M$. Denote by $l_M(G)$ the minimum number of copies of $M$ which need to be added to $G$ such that the resulting graph is Class I. If such a number does not exist for $M$, we set $l_M(G)=+\infty$.
\end{definition}

Lewis Carroll has already been a great source of graph theoretical jargon, especially when dealing with snarks: with words like \emph{``boojum"} \cite{Gardner,TutteBoojum} and \emph{``bandersnatch"} \cite{SeymourBandersnatch} used to represent snarks or graphs having some particular property. Below, we shall study what we believe is another \emph{``unmistakable"} characteristic of snarks so much so to deserve another Carrollian word which captures this bizarre behaviour. Consequently, we shall say that a bridgeless cubic graph $G$ is \emph{frumious}\footnote[2]{Coined by Lewis Carroll and was first used in his poem \emph{Jabberwocky}. It is the blend of fuming and furious.} if $l_M(G)=+\infty$ for all perfect matchings $M$ of $G$, and we will conjecture that frumious snarks are exactly the snarks in ${\cal S}_{\geq5}$ (see Conjecture \ref{conj:G+tM} for a slightly stronger statement).

In the following sections we give some results on the three parameters just defined: $l(G), l_M(G)$ and $scc(G)$, and show that the class ${\cal S}_{\geq 5}$ seems to be critical in the study of all of them. More precisely, in Section \ref{Section l(G)} we determine which bridgeless cubic graphs $G$ admit a finite value for $l(G)$, and conclude that in some sense the Petersen graph is the only obstruction for this parameter to be finite. We also show that this parameter can be arbitrarily large (see Corollary \ref{Corollary FiniteArbLarge}). In Section \ref{Section lM(G)} we conjecture that there is no snark $G$ and a perfect matching $M$ of $G$ for which $1<l_{M}(G)<+\infty$. In Theorem \ref{TheoremFlower} we also show that there exist snarks in ${\cal S}_{4}$ which are closer to being Class I than other snarks in ${\cal S}_{4}$: we show that $l_{M}(G)=1$ for any flower snark $G$ and for any perfect matching $M$ of $G$, except the Tietze graph. In Section \ref{Section scc} we show that given a bridgeless cubic graph $G$, $scc(G)$ is equal to $\sfrac{4}{3}\cdot|E(G)|$ if and only if there exists a perfect matching $M$ of $G$ for which $l_{M}(G)$ is finite. In particular, extending a result in \cite{WindmillEspMazz}, we prove that the graphs in an infinite family of snarks in ${\cal S}_{5}$ (treelike snarks) admit a shortest cycle cover whose length is strictly greater than $\sfrac{4}{3}$ their size.

\section{Notation and definitions}

The vertex set and edge set of a given graph $G$ will be denoted by $V(G)$ and $E(G)$, respectively. In what follows, graphs may contain parallel edges but no loops. A graph is said to be \emph{simple} if it does not contain any parallel edges.
A \emph{$k$--factor} of $G$, for some positive integer $k$, is a $k$--regular spanning subgraph of $G$, not necessarily connected. In particular, a \emph{perfect matching} is the edge-set of a 1--factor. A \emph{circuit} is a connected 2--regular subgraph, and a \emph{cycle} is an even subgraph of the graph. Observe that when the graph is cubic, any cycle is a collection of vertex-disjoint circuits.

A graph $G$ is said to be $k$--edge-connected if the cardinality of the smallest edge-cut of $G$ is at least $k$. We shall refer to $2$--edge-connected graphs as \emph{bridgeless}. A graph $G$ is \emph{cyclically $k$--edge-connected} if no set of fewer than $k$ edges separates two circuits of $G$. The largest integer $k$ for which $G$ is \emph{cyclically $k$--edge-connected} is the \emph{cyclic connectivity} of $G$ (apart for three small graphs that do not have a cycle-separating edge-cut and for which the cyclic connectivity is defined as their rank).

Let $G$ be a bridgeless cubic graph having a $2$--edge-cut $X$. A $2$--\emph{edge-reduction} on $X$ is the graph operation on $G$ which creates two new smaller bridgeless cubic graphs by joining the degree two vertices in each component of $G-X$ by an edge.
Moreover, for a bridgeless cubic graph $G$ having a $3$--edge-cut $X$, a \emph{$3$--edge-reduction} on $X$ is the graph operation on $G$ which creates two new bridgeless cubic graphs by introducing a new vertex to each of the components of $G-X$ and joining it to the degree two vertices in the respective component.

In the opposite direction, we define the following standard operation on bridgeless cubic graphs. Let $G_{1}$ and $G_{2}$ be two bridgeless cubic graphs, with $v_{1}\in V(G_{1})$ and $v_{2}\in V(G_{2})$ such that the vertices adjacent to $v_{1}$ are $x_{1},y_{1},z_{1}$, and those adjacent to $v_{2}$ are $x_{2},y_{2},z_{2}$. A \emph{$3$--cut-connection} on $v_{1}$ and $v_{2}$ is a graph operation that consists of constructing the new graph $[G_{1}-v_{1}]\cup[G_{2}-v_{2}]\cup\{x_{1}x_{2}, y_{1}y_{2}, z_{1}z_{2}\}$. The 3--edge-cut $\{x_{1}x_{2}, y_{1}y_{2}, z_{1}z_{2}\}$ is referred as the \emph{principal 3--edge-cut} (see for instance \cite{FouquetVanherpe}). Analogously, we can define the \emph{$2$--cut-connection} of two bridgeless cubic graphs on two of their edges.

Unlike 3--edge-reductions, more than one graph can be obtained by a 3--cut-connection on $v_{1}$ and $v_{2}$. Unless otherwise stated, if it is not important how the adjacencies in the principal 3--edge-cut look like, we just say that the resulting graph was obtained by a $3$--cut-connection on $v_{1}$ and $v_{2}$. It is clear that any resulting graph is also bridgeless and cubic.

Let $U,W \subseteq V(G)$ such that $U\cap W=\emptyset$. The set 	 consisting of all the edges having exactly one endvertex in $U$ and one endvertex in $W$ is denoted by $[U,W]$. When $W$ is equal to $V(G)-U$, the set $[U,W]$ is denoted by $\partial_{G}U$, or equivalently $\partial_{G}W$, and when it is obvious to which graph $G$ we are referring we just write $\partial U$.
A \emph{dangling edge} is an edge having exactly one end-vertex. The graph on $U$ whose edge set consists of those edges of $G$ having both endvertices in $U$ is denoted by $G[U]$. The latter is referred to as the \emph{induced} subgraph of $G$ on $U$. The subgraph of $G$ with vertex set $U$ resulting by considering $G[U]$ together with $\partial U$ as dangling edges, is said to be a \emph{$k$--pole}, with $k=|\partial U|$. A dangling edge with end-vertex $x$ is said to be \emph{joined} to a vertex $y$, if the dangling edge is deleted and $x$ and $y$ are made adjacent. In a similar way, two dangling edges are \emph{joined} if they are both deleted and their end-vertices are made adjacent.

A \emph{$k$--cycle cover} is a cycle cover consisting of at most $k$ cycles. A cycle cover $\mathcal{C}$ of a graph $G$ is said to be a \emph{cycle double cover} if every edge of $G$ is contained in exactly two cycles of $\mathcal{C}$.

Let $S$ be a finite set of colours containing at least two distinct colours $a$ and $b$. In an edge-colouring of $E(G)$, if $e$ is an edge assigned colour $a$, the $(a,b)$\emph{--Kempe chain} of $G$ containing $e$ is the maximal connected subset of $E(G)$ which contains $e$ and whose edges are all coloured either $a$ or $b$.

\section{The parameter $l(G)$}\label{Section l(G)}

We recall that for a graph $G$, $l(G)$ denotes the minimum number of perfect matchings needed to be added to $G$ in order to obtain a Class I graph. This section has two aims: to derive a sufficient condition for a bridgeless cubic graph $G$ for which $l(G)$ is finite (see Lemma \ref{LemmaCyc4Lat}) and, in such a case, to show that $l(G)$ can be arbitrarily large (see Proposition \ref{prop:lHn} and Corollary \ref{Corollary FiniteArbLarge}). Along the entire section, let $G$ be a bridgeless cubic graph. As already mentioned, $l(G)=0$ if and only if $G$ is 3--edge-colourable. Another easy observation is the following:

\begin{proposition}\label{PropositionG+MClass1}
For every bridgeless cubic graph $G$, $l(G)=1$ if and only if $\chi_{e}'(G)=4$.
\end{proposition}
\begin{proof}
If $l(G)=1$, then $G$ admits a perfect matching, say $M$, such that $G+M=F_{1}+F_{2}+F_{3}+F_{4}$, where each $F_{i}$ is a perfect matching of $G$. Clearly, $E(G)= \cup_{i=1}^{4}F_{i}$, implying that $\chi_{e}'(G)\leq 4$. Since $G$ is not itself Class I, $\chi_{e}'(G)=4$. Conversely, assume $\chi_{e}'(G)=4$. Consequently, $E(G)= \cup_{i=1}^{4}M_{i}$, for some perfect matchings $M_{i}$ of $G$. Each edge of $G$ belongs to exactly one or two of these four perfect matchings. The edges belonging to exactly two of these perfect matchings induce a perfect matching which we denote by $M$. Since $G+M=M_{1}+M_{2}+M_{3}+M_{4}$, we have $l(G)=1$.
\end{proof}

By the above, we have $l(G)>1$ if and only if $\chi'_e(G)\geq 5$. In what follows, we analyse the behaviour of $l(G)$ in the class ${\cal S}_{\geq 5}$. We start with the smallest bridgeless cubic graph having perfect matching index equal to $5$: the Petersen graph $P$. In some sense, we shall prove that the Petersen graph is the unique obstruction for a graph $G$ to have a finite value for $l(G)$.

We start with a simple characterisation of graphs that meet the conditions of the original problem proposed, in which the notion of the perfect matching lattice is used. A graph $G$ is \emph{matching covered} if any edge of $G$ lies in a perfect matching of $G$. For a perfect matching $M$ of $G$, let $\chi^{M}$ be its characteristic vector, i.e. for any $e\in E(G)$: 
\[\chi^{M}(e)=
\begin{cases}
    1& \text{if } e\in M,\\
    0& \text{otherwise}.
\end{cases}\]

The \emph{perfect matching lattice} $Lat(G)$ of a matching covered graph $G$ is defined as the set of all $|E(G)|$--dimensional integral vectors over $\mathbb{Z}$ that can be represented as a sum or difference of characteristic vectors of some perfect matchings of $G$. In other words, for a vector $w\in \mathbb{Z}^{|E(G)|}$, we have $w\in Lat(G)$ if and only if $G$ admits perfect matchings $J_1,...,J_s$ and $N_1,...,N_t$, such that
\[
\Vec{w}=\chi^{J_1}+...+\chi^{J_s}-\chi^{N_1}-...-\chi^{N_t}.
\]
Let $\Vec{1}$ be the $|E(G)|$--dimensional vector whose coordinates are all $1$.

\begin{proposition}\label{PropositionMainLattice}
For a bridgeless cubic graph $G$, $l(G) < +\infty$ if and only if $\Vec{1}\in Lat(G)$.
\end{proposition}

\begin{proof} Assume that $l(G)=k$. Hence $G+M_1+\ldots+M_k$ is Class I, for some $k$ perfect matchings $M_1,\ldots,M_k$ of $G$. Consequently, $G$ admits $k+3$ perfect matchings $F_1,...,F_{k+3}$ which partition the edge set of $G+M_1+...+M_k$. One can easily see that
\[\Vec{1}=\chi^{F_1}+...+\chi^{F_{k+3}}-\chi^{M_1}-...-\chi^{M_k},\]
as required. Conversely, assume that
\[\Vec{1}=\chi^{J_1}+...+\chi^{J_s}-\chi^{N_1}-...-\chi^{N_t},\]
for some perfect matchings $J_1,...,J_s$ and $N_1,...,N_t$ of $G$ and some integers $s,t\geq 0$. Since $G$ is cubic, $s$ must be equal to $t+3$. It is not hard to see that the perfect matchings $J_1,...,J_s$ partition the edge set of $G+N_1+...+N_t$. Hence $l(G)\leq t$.
\end{proof}

The above proposition allows us to construct an example of a bridgeless cubic graph $G$ for which $l(G)= +\infty$. As one can expect, this is the Petersen graph, and the proof follows from \cite{Lov1987} (see also \cite{CLM02,CLM03}).

\begin{proposition}
\label{PropositionLatPetersen} If $P$ is the Petersen graph, then $\Vec{1}\notin Lat(P)$.
\end{proposition}


\subsection{Graphs with $l(G)$ infinite}

Next we characterise bridgeless cubic graphs $G$ for which $l(G)= +\infty$, according to their edge-connectivity. We can assume that $G$ is connected, for if $G$ is comprised of components $G_1,...,G_t$, for some integer $t>1$, then $\Vec{1}\in Lat(G)$ if and only if $\Vec{1}\in Lat(G_i)$, for all $i=1,...,t$. First we consider graphs having 2--edge-cuts.

\begin{lemma} \label{Lemma2cutsLat}
Let $G$ be a bridgeless cubic graph having a $2$--edge-cut $X$. Let $G_{1}$ and $G_{2}$ be the two bridgeless cubic graphs obtained by applying a $2$--edge-reduction on $X$. Then, $\Vec{1}\in Lat(G)$ if and only if $\Vec{1}\in Lat(G_1)$ and $\Vec{1}\in Lat(G_2)$.
\end{lemma}

\begin{proof} Let $X=\{e_1,e_2\}$ and let the new edges in $G_1$ and $G_2$ be denoted by $f_{1}$ and $f_{2}$, respectively. First assume that $\Vec{1}\in Lat(G)$. Any perfect matching $M$ of $G$ contains either both or none of the edges of $X$. In the former case, $M$ gives rise to a perfect matching of $G_i$ by simply adding $f_{i}$ to $M\cap E(G_{i})$, for $i=1,2$. Otherwise, $M\cap E(G_{i})$ is a perfect matching of $G_{i}$. By using this idea and considering the new perfect matchings of $G_{1}$ and $G_{2}$ obtained from the list of perfect matchings of $G$ whose sum and difference of their characteristic vectors give $\Vec{1}\in\mathbb{Z}^{|E(G)|}$, one can easily show that $\Vec{1}\in Lat(G_1)$ and $\Vec{1}\in Lat(G_2)$, as required.

Conversely, assume that $\Vec{1}\in Lat(G_1)$ and $\Vec{1}\in Lat(G_2)$. Then, $G_{1}$ admits two sets of perfect matchings $\mathcal{J}_1=\{J_1^{(1)},...,J_{s+3}^{(1)}\}$ and $\mathcal{N}_1=\{N_1^{(1)},...,N_{s}^{(1)}\}$ such that, $\Vec{1}\in\mathbb{Z}^{|E(G_1)|}$ can be represented as $\sum_{J\in\mathcal{J}_{1}}\chi^{J}-\sum_{N\in\mathcal{N}_{1}}\chi^{N}$, for some integer $s\geq 0$. Similarly, $G_{2}$ admits two sets of perfect matchings $\mathcal{J}_2=\{J_1^{(2)},...,J_{t+3}^{(2)}\}$ and $\mathcal{N}_2=\{N_1^{(2)},...,N_{t}^{(2)}\}$ such that, $\Vec{1}\in\mathbb{Z}^{|E(G_2)|}$ can be represented as $\sum_{J\in\mathcal{J}_{2}}\chi^{J}-\sum_{N\in\mathcal{N}_{2}}\chi^{N}$, for some integer $t\geq 0$. The number of perfect matchings in $\mathcal{J}_1\cup \mathcal{N}_1$ which contain $f_{1}$ is odd, and is denoted by $2s'+1$, for some integer $s'\geq 0$. Moreover, the number of perfect matchings containing $f_{1}$ in $\mathcal{J}_1$ is one more than the number of such perfect matchings in $\mathcal{N}_1$. The same applies for $G_2$, and, in this case, we denote the total number of perfect matchings in $\mathcal{J}_2\cup \mathcal{N}_2$ which contain $f_{2}$ by $2t'+1$, for some integer $t'\geq 0$.

We can further assume that $2s'+1=2t'+1$, for, suppose that $s'<t'$, without loss of generality. By taking any perfect matching $F$ of $G_1$ containing $f_1$ (the existence is guaranteed by \cite{Sch}), it is easy to see that \[\sum_{J\in\mathcal{J}_{1}}\chi^{J}+\sum_{i=1}^{t'-s'}\chi^{F}-\sum_{N\in\mathcal{N}_{1}}\chi^{N}-\sum_{i=1}^{t'-s'}\chi^{F}=\Vec{1}\in\mathbb{Z}^{|E(G_{1})|}.\] Consequently, a new list of perfect matchings of $G_{1}$ whose characteristic vectors give $\Vec{1}\in\mathbb{Z}^{|E(G_1)|}$ is obtained. Moreover, exactly $2t'+1$ perfect matchings from this list contain the edge $f_{1}$, as required, and so we can assume that $s'=t'$. By a similar reasoning we can assume that $s=t$.

Without loss of generality, let the first $s'+1$ perfect matchings in $\mathcal{J}_1$ ($\mathcal{J}_2$) and the first $s'$ perfect matchings in $\mathcal{N}_1$ ($\mathcal{N}_2$) contain $f_{1}$ ($f_{2}$). Let $\mathcal{J}=\{J_{1},\ldots, J_{s+3}\}$, where
\[J_i=
\begin{cases}
    (J_i^{(1)}-f_{1})\cup (J_i^{(2)}-f_{2})\cup\{e_{1},e_{2}\}& \text{if } i=1,\ldots,s'+1,\\
    J_i^{(1)}\cup J_i^{(2)}& \text{otherwise.}
\end{cases}\]
Similarly, let $\mathcal{N}=\{N_{1},\ldots, N_{s}\}$, where
\[N_i=
\begin{cases}
    (N_i^{(1)}-f_{1})\cup (N_i^{(2)}-f_{2})\cup\{e_{1},e_{2}\}& \text{if } i=1,\ldots,s',\\
    N_i^{(1)}\cup N_i^{(2)}& \text{otherwise.}
\end{cases}\]
One can see that $\mathcal{J}$ and $\mathcal{N}$ are two sets consisting of perfect matchings of $G$, such that $\sum_{J\in\mathcal{J}}\chi^{J}-\sum_{N\in\mathcal{N}}\chi^{N}=\Vec{1}\in\mathbb{Z}^{E(G)},$ as required.
\end{proof} 

The proved statement suggests that $l(G)=+\infty$ if and only if $l(G_1)=+\infty$ or $l(G_2)=+\infty$. Thus, in trying to characterise the bridgeless cubic graphs $G$ with $l(G)=+\infty$, one can focus on 3--edge-connected graphs having 3--edge-cuts. Following \cite{Lov1987}, we say that an edge-cut in $G$ is \emph{tight} if any perfect matching of $G$ intersects it in exactly one edge (not necessarily the same).

\begin{lemma}
\label{Lemma3cutsLat}
Let $G$ be a $3$--edge-connected cubic graph and let $X$ be a non-trivial tight $3$--edge-cut in $G$. Consider the two bridgeless cubic graphs $G_1$ and $G_{2}$ obtained by applying a 3--edge-reduction to $X$. Then, $\Vec{1}\in Lat(G)$ if and only if $\Vec{1}\in Lat(G_1)$ and $\Vec{1}\in Lat(G_2)$.
\end{lemma}

This statement can be derived from the results of \cite{Lov1987}. Moreover, its proof follows an argument similar to the one used in the proof of Lemma \ref{Lemma2cutsLat}. For these reasons we omit the proof here.

Before we proceed to prove the next result regarding 3--edge-connected cubic graphs which do not contain non-trivial tight 3--edge-cuts we give the definition of a brick.
A \emph{brick} is a $3$--connected graph such that for any two distinct vertices $u$ and $v$ of $G$, $G-u-v$ admits a perfect matching. It is easy to see that no brick can be bipartite.

\begin{lemma} \label{LemmaCyc4Lat}
Let $G$ be a $3$--edge-connected cubic graph without non-trivial tight $3$--edge-cuts. Then, $\Vec{1}\in Lat(G)$ if and only if $G$ is not the Petersen graph.
\end{lemma}

\begin{proof} If $G$ is the Petersen graph, then by Proposition \ref{PropositionLatPetersen}, $\Vec{1}\notin Lat(G)$. So assume that $G$ is not the Petersen graph. Since $G$ is cubic, by \cite{KCLL18} we have that all tight edge-cuts of $G$ are 3--edge-cuts. Thus, by our assumptions, $G$ contains no tight edge-cuts. Hence, by \cite{Lov1987}, $G$ is either bipartite or a brick. Now, if $G$ is bipartite, then it is $3$--edge-colourable and so $\Vec{1}\in Lat(G)$. Hence, we can assume that $G$ is a brick. The main result of \cite{Lov1987} implies that the only cubic brick for which $\Vec{1}\notin Lat(G)$ is the Petersen graph, proving our result.
\end{proof}

\begin{corollary}\label{Corollary Cyc4Conn NOT Petersen l Finite}
Let $G$ be a cyclically $4$--edge-connected cubic graph different from the Petersen graph. Then, $l(G)$ is finite.
\end{corollary}

\subsection{Construction of cubic graphs with $l(G)$ finite but arbitrarily large}

We have already seen that $l(G)\leq 1$ if and only if $\chi_{e}'(G)\leq 4$. The results obtained above suggest an algorithm to check whether $l(G)=+\infty$ for a given bridgeless cubic graph $G$. The next question that we would like to address is to see whether there exist graphs in ${\cal S}_{\geq5}$ with $1<l(G)<\infty$. In Corollary \ref{Corollary FiniteArbLarge}, we show that there exist bridgeless cubic graphs $G$ with $l(G)$ finite but arbitrarily large.

Let $G$ be a bipartite graph with bipartition $U$ and $W$. Let $u\in U$. We say that $G$ is \emph{coverable with respect to $u$} if for every $w\in W$ there exists a parity subgraph of $G$ in which the vertices $u$ and $w$ are of degree 3 and all the other vertices are of degree 1. We remark that a \emph{parity subgraph} of $G$ is a spanning subgraph of $G$ with the degrees of all the vertices having the same parity in both the subgraph and in $G$.

Let $G$ be a bipartite cubic graph of order $2n$ having bipartition $U$ and $W$. Assume $W=\{w_1, w_2,\ldots,w_n\}$ and let $u\in U$. Let $v$ be a vertex of the Petersen graph $P$, and let $P_{1}, \ldots, P_{n}$ be $n$ copies of the Petersen graph, with the vertex corresponding to $v$ in each copy denoted by $v_{1}, \ldots, v_{n}$, respectively. Apply a 3--cut-connection on $v_{i}$ and $w_{i}$, for each $i\in\{1,\ldots,n\}$ and expand $u$ to a triangle. The resulting graph will be called an \emph{extension of $G$ with respect to $u$}.

\begin{proposition} \label{prop:lHn}
Let $G$ be a bipartite cubic graph of order $2n$ and let $H$ be an extension of $G$ with respect to $u$, for some $u\in V(G)$. If $G$ is coverable with respect to $u$, then $l(H)=n$.
\end{proposition}
\begin{proof}

We claim that $l(H)\geq n$. Suppose that $l(H)=k<n$, for contradiction. Then, $H$ admits $k$ perfect matchings $M_1,...,M_{k}$, such that $H+M_1+...+M_{k}$ is Class I. Since $G$ is bipartite, if a perfect matching $M$ of $H$ intersects all the three edges of $\partial(P_{i}-v_{i})$ in $H$, for some $i\in\{1,\ldots,n\}$, then, $|M\cap \partial(P_{j}-v_{j})|=1$, for all $j\in\{1,\ldots,i-1,i+1,\ldots, n\}$. In this case, $M$ must also intersect the three edges incident with the triangle in $H$. Since $k\leq n-1$, there exists some $s\in\{1,\ldots,n\}$ such that $\partial(P_{s}-v_{s})$ is not contained in any perfect matching in $M_1,...,M_{k}$. Thus, these perfect matchings intersect exactly one edge from the $3$--edge-cut $\partial(P_{s}-v_{s})$. Hence, $M_{1},\ldots, M_{k}$ induce $k$ perfect matchings of the Petersen graph ($P_{s}$), say $M_{1}', \ldots, M_{k}'$.
Let $F_{1}, \ldots, F_{k+3}$ be the $k+3$ colours of $H+M_{1}+\ldots+M_{k}$. By a simple counting argument, $|F_{i}\cap\partial(P_{s}-v_{s})|=1$, for each $i\in\{1,\ldots,k+3\}$. Therefore, the $F_{i}$s induce $k+3$ perfect matchings of the Petersen graph ($P_{s}$), say $F_{1}',\ldots, F_{k+3}'$. However, this implies that $P_{s}+M_{1}'+\ldots+M_{k}'=F_{1}'+\ldots+F_{k+3}'$, a contradiction to Proposition \ref{PropositionLatPetersen}.

\begin{figure}[h]
      \centering
      \includegraphics[width=0.55\textwidth]{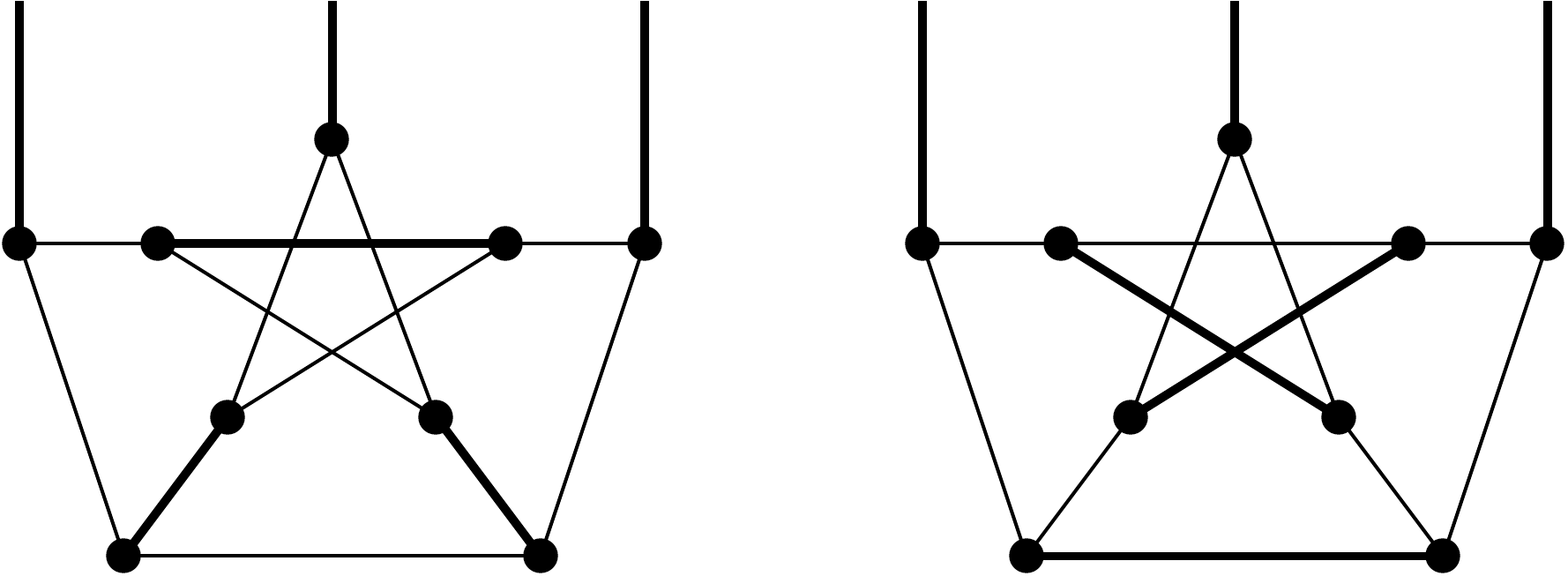}
      \caption{The way $N_{i}$ and $J_{i}$ intersect $P_{i}$}
      \label{FigurePrismPM}
\end{figure}

Now, we show that $l(H)$ is actually equal to $n$.
For each $i\in\{1,\ldots,n\}$, let $N_{i}$ be a perfect matching of $H$ containing $\partial(P_{i}-v_{i})$ and intersecting $P_{i}-v_{i}$ as depicted on the left in Figure \ref{FigurePrismPM}. Since $G$ is coverable with respect to $u$, such a perfect matching exists.  We claim that $H+N_{1}+\ldots+N_{n}$ is Class I. For each $i\in\{1,\ldots,n\}$, let $J_{i}$ be the perfect matching of $H$ equal to $N_{i}$, apart from the way it intersects the edges in $P_{i}-v_{i}$. One can see the differences in Figure \ref{FigurePrismPM}.

\begin{figure}[h]
      \centering
      \includegraphics[width=0.24\textwidth]{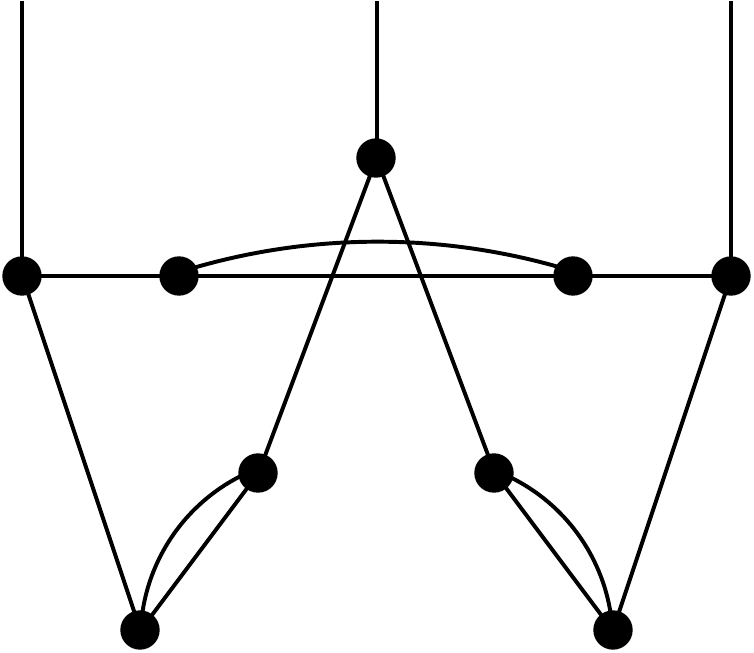}
      \caption{$P_{i}'$ in $H+\sum_{i=1}^{n}(N_{i}-J_{i}$)}
      \label{FigurePrism3Col}
\end{figure}

Consider the graph $H+\sum_{i=1}^{n}(N_{i}-J_{i})$. This will have the same structure as $H$, however, every $P_{i}-v_{i}$ is now transformed into $P_{i}'$, as shown in Figure \ref{FigurePrism3Col}. Since a bipartite graph is Class I and the 3--pole $P_{i}'$ can be 3--edge-coloured in such a way that its three dangling edges each have a different colour, then, a 3-edge-colouring of $G$ can be easily extended to a 3--edge-colouring of $H+\sum_{i=1}^{n}N_{i}-J_{i}$. Let these three colours (also perfect matchings of $H$) be denoted by $J_{n+1}, J_{n+2},J_{n+3}$. Consequently, $H+N_{1}+\ldots+N_{n}=J_{1}+\ldots+J_{n+3}$, implying that $l(H)=n$, as required.
\end{proof}

We remark that the above result holds also for bipartite cubic graphs $G$ admitting parallel edges. Moreover, by using Proposition \ref{prop:lHn} we have the following consequence:

\begin{corollary}\label{Corollary FiniteArbLarge}
For each positive integer $n$ there exists a cubic graph $H$ with $l(H)=n$.
\end{corollary}
\begin{proof}
Every snark having perfect matching index $4$ is an example for $n=1$. Moreover, we directly checked that the value of $l$ for the (treelike) snark on 34 vertices, also known as \emph{windmill} (see \cite{AbreuEtAl, WindmillEspMazz}), is $2$. For $n>2$, it can be observed that if $G$ is the circular ladder graph on $2n$ vertices (if $n$ is even) or the  M\"{o}bius ladder graph on $2n$ vertices (if $n$ is odd), then for any vertex $u\in V(G)$, $G$ is coverable with respect to $u$. Thus, the result follows from Proposition~\ref{prop:lHn}.
\end{proof}

Finally, the following natural question arises:

\begin{problem} \label{ProblemCyc4ArbLargeL}
Does there exist a cyclically 4--edge-connected cubic graph with arbitrarily large $l$?
\end{problem}

We recall that $l$ is always finite in the class of cyclically 4--edge-connected cubic graphs excluding the Petersen graph, as the latter is the only cyclically 4--edge-connected cubic graph for which $l$ is infinite by Corollary \ref{Corollary Cyc4Conn NOT Petersen l Finite}.

\section{The parameter $l_M(G)$} \label{Section lM(G)}

Proposition \ref{PropositionG+MClass1} states that if $G$ belongs to ${\cal S}_4$, then it admits a perfect matching which when added to $G$ the resulting graph is Class I.
What happens if $G$ belongs to ${\cal S}_{\geq5}$? For sure, for any perfect matching $M$ of $G$, $G+M$ is not Class I. However, what can we say about $G+tM$, for $t$ being a positive integer strictly greater than 1?

We recall that the parameter $l_M(G)$, for a given bridgeless cubic graph $G$ and a given perfect matching $M$ of $G$, is defined as the minimum $t$, if such an integer exists, for which $G+tM$ is Class I.
Clearly, $l_M(G)\geq l(G)$ for every perfect matching $M$ of $G$, and thus, if $l(G)=+\infty$ then $l_M(G)=+\infty$ for every perfect matching $M$ of $G$. Up till now, we are not able to find any pair $(G,M)$ such that $1<l_M(G)<+\infty$, and we are inclined to believe that such an example does not exist at all.

\begin{conjecture}\label{conj:G+tM}
If $G+M$ is Class II for a given perfect matching $M$ of $G$, then $G+tM$ is Class II for all positive integers $t$.
\end{conjecture}

Trivially, $G$ is Class I if and only if $l_M(G)=0$ for every perfect matching $M$ of $G$. Moreover, $G \in {\cal S}_4$ if and only if $l(G)=1$.
Last assertion and Conjecture \ref{conj:G+tM} would imply that if $G \in {\cal S}_{\geq 5}$, then $G$ is frumious, whilst if $G \in {\cal S}_4$, then $l_M(G)$ is equal to 1 or $+\infty$ according to the selected perfect matching $M$.

The class ${\cal S}_4$ can be considered as the class of bridgeless cubic graphs closest to the class of 3--edge-colourable cubic graphs. Previous considerations suggest that there could be graphs inside ${\cal S}_4$ which are closer to being 3--edge-colourable than others: these are Class II bridgeless cubic graphs $G$ for which $G+M$ is Class I for any one of their perfect matchings $M$, i.e. $l_M(G)=1$ for every $M$.
We cannot give a complete characterisation of the graphs which have this property. However, we are able to show that an infinite family of snarks, with perfect matching index four (shown in \cite{FouquetVanherpe}), have this distinctive property.

\subsection{Examples of cubic graphs $G$ such that $l_{M}(G)=1$ for every $M$}

\begin{definition}
The 6--pole on four vertices shown in Figure \ref{FigureSFDF} will be called \textit{a Single-Flower 6--pole}, for short an \textit{SF 6--pole}, whilst its vertical edge will be referred to as a \textit{spoke}.

\end{definition}

\begin{figure}[H]
      \centering
      \includegraphics[width=0.23\textwidth]{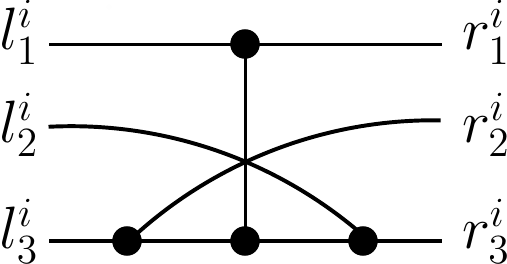}
      \caption{The SF 6--pole $F_{i}$}
      \label{FigureSFDF}
\end{figure}

Let $n\geq 3$ be an odd integer, and let $F_{1}, F_{2}, \ldots, F_{n}$ be $n$ SF 6--poles. Let $l_1^i,l_2^i,l_3^i$ and $r_1^i,r_2^i,r_3^i$ be the \emph{left} and \emph{right} dangling edges of $F_{i}$, respectively, as shown in Figure \ref{FigureSFDF}. The graph obtained by joining the dangling edges $r_{j}^i$ and $l_{j}^{i+1}$, for every $i\in\{1,\ldots,n\}$ and for every $j\in\{1,2,3\}$, is called a \emph{flower snark} and is denoted by $\mathcal{F}_{n}$ (see \cite{Isaacs}). We remark that all operations in the upper indexing set are taken modulo $n$. The new edge obtained after joining two dangling edges, say $r_{j}^i$ and $l_{j}^{i+1}$, will be referred to interchangeably by the same two names.
To simplify the way we depict flower snarks, we shall look at $\mathcal{F}_{n}$ as a 6--pole with the left and right dangling edges being $l_1^1,l_2^1,l_3^1$, and $r_1^n,r_2^n,r_3^n$, respectively.

The 6--pole obtained by joining the right dangling edges of an SF 6--pole with the left dangling edges of another SF 6--pole in the same way as in the construction of flower snarks is called a \textit{Double-Flower 6--pole}, for short a \textit{DF 6--pole} (see Figure \ref{Figure6pole+M}).

\begin{figure}[h]
      \centering
      \includegraphics[width=0.8\textwidth]{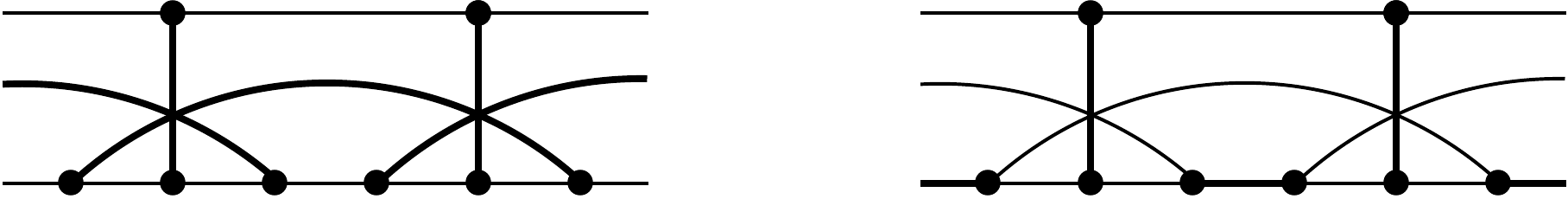}
      \caption{DF 6--poles in $\mathcal{F}_{n}$ with two consecutive spokes belonging to a perfect matching}
      \label{Figure6pole+M}
\end{figure}

\begin{definition}
Let $X$ be a DF 6--pole in $\mathcal{F}_{n}$, with left and right dangling edges $l_1^{i},l_2^{i},l_3^{i}$ and $r_1^{i+1},r_2^{i+1},r_3^{i+1}$, respectively, for some $i\in \{1,\ldots,n\}$, and let $M$ be a perfect matching of $\mathcal{F}_{n}$. The DF 6--pole $X$ is said to be \textit{good with respect $M$}, if 
there exists $j \in \{1,2,3\}$ such that $\partial X\cap M=\{l_j^{i},r_j^{i+1}\}$.
\end{definition}

In the sequel, we prove that given a perfect matching $M$ of $\mathcal{F}_{n}$, $\mathcal{F}_{n}+M$ is Class I, except when $n=3$ and $M$ intersects exactly one spoke of $\mathcal{F}_{3}$.
The latter case arises because the graph $\mathcal{F}_{3}$ is the Petersen graph $P$ with one vertex expanded to a triangle (also known as the Tietze graph), and if $\mathcal{F}_{3}+M$ is Class I for such a perfect matching $M$, then this would imply that $l(P)=1$, a contradiction (see Proposition \ref{PropositionMainLattice} and Proposition \ref{PropositionLatPetersen}). \\

Easy direct checks show that the following remarks hold:

\begin{enumerate}[label=\textbf{R.\arabic*},ref=R.\arabic*]
\item \label{Remark1_J3+M3} Let $M$ be a perfect matching of $\mathcal{F}_{3}$ intersecting all three of its spokes. Then, $\mathcal{F}_{3}+M$ is Class I.
\item \label{Remark2_J5+M1} Let $M$ be a perfect matching of $\mathcal{F}_{5}$ intersecting exactly one spoke, say the spoke of $F_{3}$. Then, $M$ contains one of the two matchings depicted in Figure \ref{FigureJ5+M1}. One can clearly see that, in any case, the colouring depicted Figure \ref{FigureJ5+M1} can always be extended to a 4--edge-colouring of $\mathcal{F}_{5}+M$ using the colours $a,b,c,d$.

\begin{figure}[h]
      \centering
      \includegraphics[width=.8\textwidth]{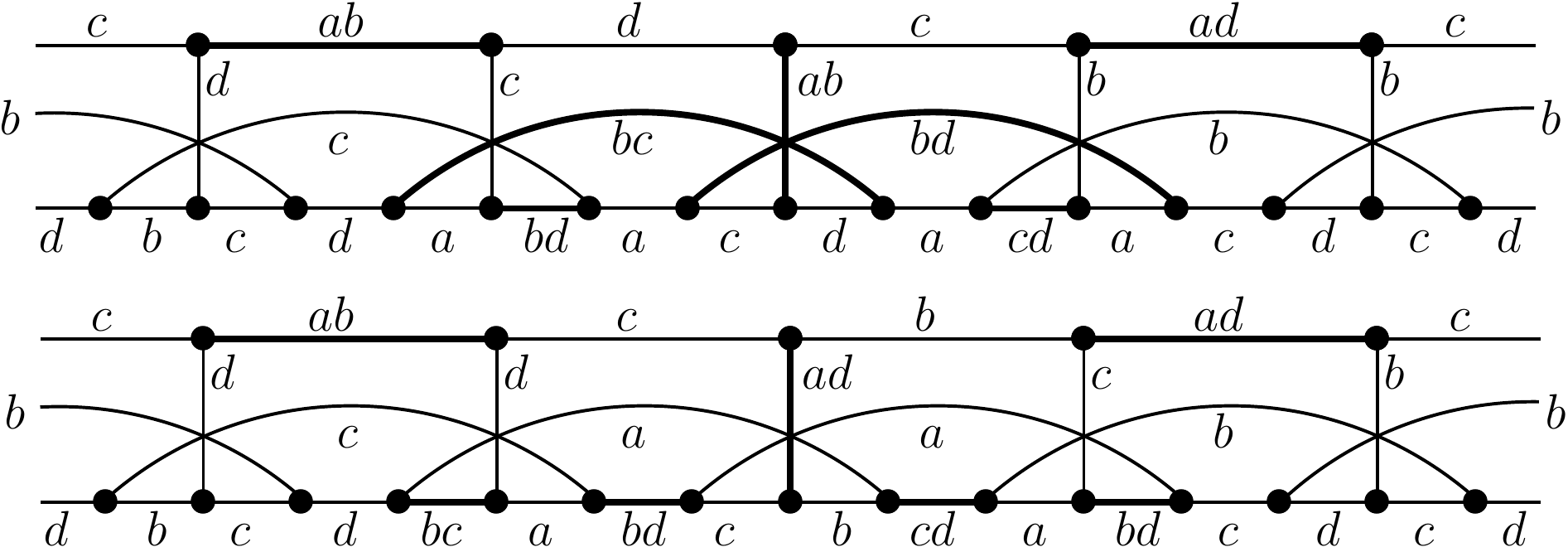}
      \caption{$M$ intersecting exactly one spoke in $\mathcal{F}_{5}$}
      \label{FigureJ5+M1}
\end{figure}
\item \label{Remark5_atmostoneinteresection} As $n$ is odd, any perfect matching of $\mathcal{F}_{n}$ intersects exactly one left (similarly right) dangling edge of some SF 6--pole $F_{i}$, for $i\in\{1,\ldots,n\}$.

Note that \ref{Remark5_atmostoneinteresection} follows because every perfect matching of $\mathcal{F}_{n}$ cannot intersect all the three left (similarly right) dangling edges of $F_{i}$. Moreover, if a perfect matching intersects exactly two left dangling edges of $F_{i}$, then the right dangling edges of this 6--pole are not intersected by the perfect matching, and vice-versa. Since $n$ is odd, this is impossible to occur.
\item \label{Remark3_ConsecutiveSpokes} If the two spokes of a DF 6--pole are contained in a perfect matching, then it is a good DF 6--pole with respect to that perfect matching (see Figure \ref{Figure6pole+M}).
\item \label{Remark4_J5+M3} If a perfect matching $M$ of $\mathcal{F}_{n}$ intersects the first and third out of three consecutive spokes, then, the second spoke must be contained in $M$, as well. Consequently, if a perfect matching of $\mathcal{F}_{5}$ intersects exactly three spokes, then they must be consecutive. Moreover, in this case, the two SF 6--poles of $\mathcal{F}_{5}$ whose spokes are not contained in the perfect matching form a good DF 6--pole.
\item \label{Remark6 nothreeconsecutivenotinM} As $n$ is odd, if the spokes of three consecutive SF 6--poles, say $F_1,F_2,F_3$, do not belong to a perfect matching, then either $F_1$ and $F_2$, or, $F_2$ and $F_3$ form a good DF 6--pole with respect to that perfect matching.

Indeed, note that either $l_1^1$ and $r_1^2$, or, $l_1^2$ and $r_1^3$ belong to the perfect matching, and so \ref{Remark6 nothreeconsecutivenotinM} follows by \ref{Remark5_atmostoneinteresection}.

\end{enumerate}

In the next proof we will make use of the following procedure: we delete a good DF 6--pole $X$ with respect to a perfect matching $M$ of $\mathcal{F}_{n}$ (shown as the dotted part in Figure \ref{Figure6poleInduction}) and join the remaining dangling edges accordingly together as in Figure \ref{Figure6poleInduction}. In this way we obtain a copy of the flower snark $\mathcal{F}_{n-2}$.

\begin{figure}[h]
      \centering
      \includegraphics[width=0.65\textwidth]{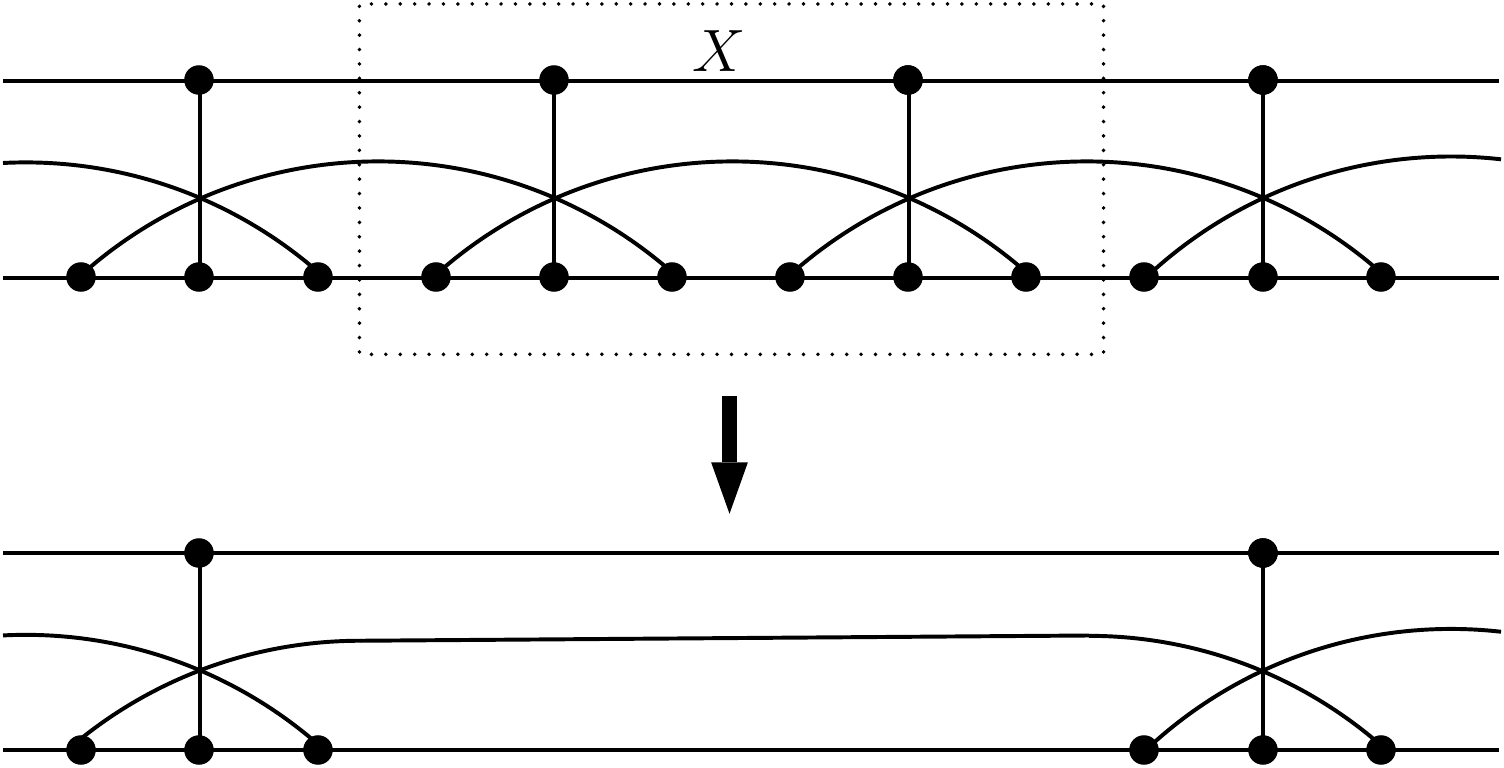}
      \caption{Inductive step in the proof of Theorem \ref{TheoremFlower}}
      \label{Figure6poleInduction}
\end{figure}

In the sequel, with a slight abuse of terminology, we shall refer to the three edges obtained after joining the above dangling edges as the new edges of $\mathcal{F}_{n-2}$.
Moreover, since $X$ is good, $M$ naturally induces a perfect matching of $\mathcal{F}_{n-2}$. We will denote by $M_X$ such a perfect matching in the copy of $\mathcal{F}_{n-2}$, obtained by removing $X$ from $\mathcal{F}_{n}$. Note that $M_X$ contains exactly one of the three new edges.

\begin{theorem}\label{TheoremFlower}
Let $n\geq 5$ be an odd integer and let $M$ be a perfect matching of $\mathcal{F}_{n}$. Then, $\mathcal{F}_{n}+M$ is Class I.
\end{theorem}

\begin{proof}

The crucial steps of the proof of this theorem lie in the following two claims.\\

{\bf Claim I:} Let $n\geq 5$ be an odd integer and let $X$ be a good DF 6--pole with respect to a perfect matching $M$ of $\mathcal{F}_{n}$.
If  $\mathcal{F}_{n-2}+M_X$ is Class I, then  $\mathcal{F}_{n}+M$ is Class I. 

\emph{Proof of Claim I.} Let $M_X$ be the perfect matching induced by $M$ in $\mathcal{F}_{n-2}$.
By assumption, $\mathcal{F}_{n-2}+M_X$ admits a $4$--edge-colouring with colours denoted by $a,b,c,d$. Without loss of generality, we can assume that the unique edge of $\mathcal{F}_{n-2}+M_X$ parallel to a new edge of $\mathcal{F}_{n-2}$ has colour $d$ in the given $4$--edge-colouring. Since every colour class corresponds to a perfect matching of $\mathcal{F}_{n-2}$, it follows by \ref{Remark5_atmostoneinteresection} that each of the colours $a,b,c$ intersects exactly one of the three new edges of $\mathcal{F}_{n-2}$.
A $4$--edge-colouring of $\mathcal{F}_{n}+M$ is constructed in the following way:
if an edge does not have an end-vertex in $X$, then it is assigned the same colour of its corresponding edge in $\mathcal{F}_{n-2}+M_X$; all edges of $\mathcal{F}_{n}$ with an end-vertex in $X$ are assigned the colours $a,b,c$ as illustrated in Figure \ref{Figure6pole3col}; and finally, all edges of $M$ with an end-vertex in $X$ are assigned the colour $d$. Since this gives rise to a $4$--edge-colouring of $\mathcal{F}_{n}+M$, the claim follows.

\begin{figure}[H]
      \centering
      \includegraphics[width=0.62\textwidth]{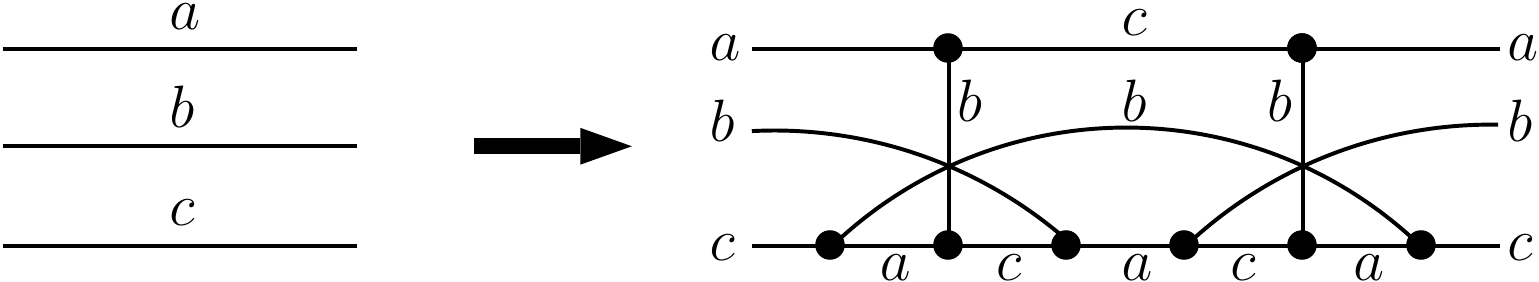}
      \caption{Extending the colour classes of $\mathcal{F}_{n-2}$ to $\mathcal{F}_{n}$}
      \label{Figure6pole3col}
\end{figure}

{\bf Claim II:} Let $n\geq 3$ be an odd integer and let $M$ be a perfect matching of $\mathcal{F}_{n}$. The graph $\mathcal{F}_{n}$ admits a good DF 6--pole with respect to $M$.

\emph{Proof of Claim II.} Suppose that there is no good DF 6--pole with respect to $M$, for contradiction. From \ref{Remark3_ConsecutiveSpokes} it follows that $M$ cannot contain two consecutive spokes. At the same time, since $n$ is odd, \ref{Remark4_J5+M3} and \ref{Remark6 nothreeconsecutivenotinM} imply that every sequence of consecutive spokes not in $M$ has length exactly two. Hence, for every three consecutive spokes, one of them belongs to $M$ and the other two do not.
Consider three consecutive SF 6--poles in $\mathcal{F}_{n}$, and without loss of generality assume that $M$ intersects only the first spoke. Since there is no good DF 6--pole with respect to $M$, a direct easy check shows that $M$ can intersect these three consecutive SF 6--poles only in two possible ways, as shown in Figure \ref{3consecutiveSF}.

\begin{figure}[h]
      \centering
      \includegraphics[width=0.95\textwidth]{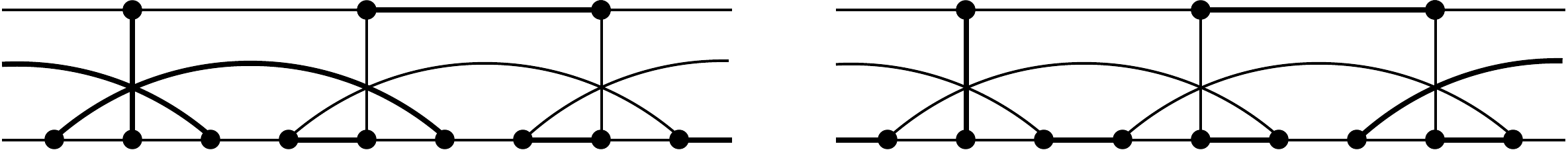}
      \caption{How $M$ can intersect three consecutive SF 6--poles}
      \label{3consecutiveSF}
\end{figure}

The two ways $M$ can intersect three consecutive SF 6--poles must alternate in $\mathcal{F}_{n}$. Hence, $n$ is three times an even number, a contradiction, since $n$ is assumed to be odd. \\

Now we are in a position to complete the proof of the theorem. We prove the result by induction on $n$. Consider first $\mathcal{F}_{5}$. As the spokes form an odd edge-cut, $M$ intersects an odd number of them. By \ref{Remark2_J5+M1} we can assume that $M$ intersects at least three consecutive spokes of $\mathcal{F}_{5}$, say the spokes of $F_{1}, F_{2},F_{3}$. Consequently, by \ref{Remark3_ConsecutiveSpokes} or \ref{Remark4_J5+M3}, $F_{4}$ and $F_{5}$ form a good DF 6--pole with respect to $M$. Let this DF 6--pole be $X$.  We have that $M_X$ intersects all the three spokes of $\mathcal{F}_{3}$. By \ref{Remark1_J3+M3}, $\mathcal{F}_{3}+M_{X}$ is Class I and the base case $n=5$ follows by Claim I.

Now, assume the result holds up to $n\geq 5$, i.e. $\mathcal{F}_{n}+M$ is Class I for every perfect matching $M$ of $\mathcal{F}_{n}$. Consider $\mathcal{F}_{n+2}$ and let $M$ be one of its perfect matchings.
By Claim II, $\mathcal{F}_{n+2}$ admits a good DF 6--pole $X$ with respect to $M$.
By induction, $\mathcal{F}_{n}+M_X$ is Class I and the assertion follows by Claim I.
\end{proof}

The flower snark $\mathcal{F}_{5}$ has cyclic connectivity 5, and for every odd $n\geq 7$, $\mathcal{F}_{n}$ has cyclic connectivity 6. Because of Theorem \ref{TheoremFlower}, one may think that for every perfect matching $M$ of a cyclically 5--edge-connected cubic graph $G$ with perfect matching index four, $G+M$ is Class I. However, this is not true. By Theorem 1.1 in \cite{HagglundOstenhof}, there exists an infinite family of cyclically 5--edge-connected cubic graphs $G$ having perfect matching index 4, which do not satisfy this assertion. This is true because these graphs admit a 2--factor which is not contained in any one of their cycle double covers. For, let $G$ be such a graph, and let $N$ be the complement of such a 2--factor $C$. Suppose that $G+N$ is Class I, for contradiction. Then, $G+N=\sum_{i=1}^{4}J_{i}$ for some perfect matchings $J_{i}$ of $G$. Hence, $\{N\triangle J_{1}, \ldots, N\triangle J_{4},C\}$ is a cycle double cover of $G$ containing $C$. This contradicts our choice of $G$. 

\section{A relation between $l_M(G)$ and $scc(G)$} \label{Section scc}

The main conjecture in the area of short cycle covers of bridgeless graphs is the so-called $\sfrac{7}{5}$--Conjecture (or the Shortest Cycle Cover Conjecture). It states that for any bridgeless graph $G$ (not necessarily cubic), we have $scc(G)\leq \sfrac{7}{5}\cdot |E(G)|$. This conjecture is one of the many consequences of the Petersen Colouring Conjecture \cite{AR}. On the other hand, it implies the Cycle Double Cover Conjecture, see~\cite{JT}. In \cite{ShortCC} it is shown that any bridgeless cubic graph $G$ has a cycle cover of length at most $\sfrac{34}{21}\cdot |E(G)|$, and any bridgeless graph $G$ of minimum degree three has a cycle cover of length at most $\sfrac{44}{27}\cdot |E(G)|$.

A \emph{$k$--cycle cover} is a cycle cover consisting of at most $k$ cycles. The following conjecture can be found as Conjecture 8.11.5 in \cite{ZhangBook}:

\begin{conjecture}\label{ConjectureZhangAtMost4Cycles}
Every bridgeless graph has a shortest $4$--cycle cover.
\end{conjecture}

Here, we propose the following conjecture and we show that it is implied by Conjecture \ref{ConjectureZhangAtMost4Cycles}.

\begin{conjecture}\label{Conjectureperfect matchingIndex4SCC}
For every bridgeless cubic graph $G$, $scc(G)=\sfrac{4}{3}\cdot|E(G)|$ if and only if $\chi_{e}'(G)\leq 4$.
\end{conjecture}

\begin{proposition}
Conjecture \ref{ConjectureZhangAtMost4Cycles} implies Conjecture \ref{Conjectureperfect matchingIndex4SCC}.
\end{proposition}
\begin{proof}  If $\chi_{e}'(G)\leq 4$, then by \cite{Steffen}, $scc(G)=\sfrac{4}{3}\cdot|E(G)|$. So assume $scc(G)=\sfrac{4}{3}\cdot|E(G)|$, and let $\mathcal{C}=\{C_{1}, \ldots, C_{k}\}$ be a cycle cover of $G$ with length $\sfrac{4}{3}\cdot|E(G)|$. Since we are assuming Conjecture~\ref{ConjectureZhangAtMost4Cycles} to be true, we can assume $k\leq 4$. Since $G$ is cubic and the length of $\mathcal{C}$ is $\sfrac{4}{3}\cdot|E(G)|$, every edge of $G$ is either covered once or twice in $\mathcal{C}$ and the edges covered twice form a perfect matching of $G$, say $M$. Let $F_{i}=C_{i}\triangle M$, for every $i=1,\ldots,k$. Since $\mathcal{C}$ is a cycle cover, the perfect matchings $F_{1},\ldots, F_{k}$ cover the edge set of $G$, implying that $\chi_{e}'(G)\leq 4$, as required.
\end{proof}

A relation between $scc(G)$ and $l_M(G)$ is clearly established by the following theorem.

\begin{theorem}\label{thm:sccvslM}
For every bridgeless cubic graph $G$, $scc(G)>\sfrac{4}{3}\cdot|E(G)|$ if and only if $G$ is frumious, otherwise $scc(G)=\sfrac{4}{3}\cdot|E(G)|$.
\end{theorem}

\begin{proof}
Assume that $G$ is not frumious, i.e. $G+tM$ is Class I for a perfect matching $M$ of $G$ and for some non-negative integer $t$. Let $F_{1}, \ldots, F_{t+3}$ be the colour classes of a $(t+3)$--edge-colouring of $G+tM$. For every $i=1,\ldots, t+3$, let $C_{i}=M\triangle F_{i}$, and let $\mathcal{C}=\{C_{1},\ldots,C_{t+3}\}$. The latter is a cycle cover of $G$. Moreover, if $e\in M$, then $e$ is covered exactly twice by the cycles in $\mathcal{C}$. Otherwise, if $e\not\in M$, then $e$ is covered exactly once by some cycle in $\mathcal{C}$. Since for any cubic graph $G$ and any cycle cover $\mathcal{C}$ of $G$, $\mathcal{C}$ has length $\sfrac{4}{3}\cdot|E(G)|$ if and only if the set of edges covered twice by $\mathcal{C}$ is a perfect matching of $G$, the result follows.
Conversely, let $\mathcal{C}=\{C_{1}, \ldots, C_{t+3}\}$ be a shortest cycle cover of $G$ of length $\sfrac{4}{3}\cdot|E(G)|$, for some integer $t$. Let $M$ be the set of edges covered exactly twice by $\mathcal{C}$, and let $F_{i}$ be equal to $M\triangle C_{i}$. By an argument similar to the first implication, one can see that $G+tM=F_{1}+\ldots+F_{t+3}$.
\end{proof}

Hence, the main consequence of Conjecture \ref{Conjectureperfect matchingIndex4SCC} is that bridgeless cubic graphs having perfect matching index at least 5 would have a shortest cycle cover strictly greater than $\sfrac{4}{3}$ their size. The problem seems to be very hard to solve. However, in the next section, we show that an infinite family of snarks $G$ with perfect matching index 5 have a shortest cycle cover strictly greater than $\sfrac{4}{3}\cdot|E(G)|$. 

\subsection{Treelike Snarks}

We recall that a bridgeless cubic graph $G$ is frumious if $l_M(G)=+\infty$ for all perfect matchings $M$ of $G$.
As already remarked, the Petersen graph is such a graph and above we conjectured (see Conjecture \ref{conj:G+tM}) that a bridgeless cubic graph is frumious if and only if its perfect matching index is at least $5$ \footnote[3]{\emph{``I have said it thrice: What I tell you three times is true."} -- {\sc Lewis Carroll}, The Hunting of the Snark}.
In order to support such a conjecture we consider an infinite family of snarks, called treelike snarks, having perfect matching index $5$ and prove that they are frumious snarks. The  family of treelike snarks was first introduced in \cite{AbreuEtAl}, but here we also refer to the more general definition of treelike snarks given in \cite{EMMS} and prove our main result (Theorem \ref{TheoremTreelikeSnarks}) in this general setting. 

In order to present such a class of snarks we need some preliminary definitions.

\begin{definition}\label{def:frumious4pole}
Let $A$ be an arbitrary $4$--pole. Partition its four dangling edges in ordered pairs, say $(l_1,l_2)$, referred to as the first and second left dangling edges, and $(r_1,r_2)$, referred to as the first and second right dangling edges. Let the end-vertices of the four dangling edges $l_{1},l_{2},r_{1},r_2$ be $u_1,u_2,v_1,v_2$, respectively. The $4$--pole $A$ is said to be \linebreak \emph{frumious} with respect to such a partition if the graph obtained by removing the four dangling edges and adding two new vertices $u$ and $v$ such that $u$ is adjacent to $v,u_1,u_2$ and $v$ is adjacent to $u,v_1,v_2$, is a frumious snark. We will refer to the latter graph as the frumious snark obtained from the $4$--pole $A$.
\end{definition}

Note that a $4$--pole could be frumious with respect to a given partition whilst it is not with respect to another one. On the other hand, a change in the order of the left dangling edges or the right dangling edges of a frumious 4--pole produces another (possibly different) frumious 4--pole. However, although this last change may produce a different frumious 4--pole, the two frumious snarks obtained from the two 4--poles are the same.
An example of a frumious 4--pole is the one obtained by removing two adjacent vertices of the Petersen graph, say $u$ and $v$, with the left dangling edges corresponding to the edges originally incident with $u$ and not $v$, and the right dangling edges corresponding to the edges originally incident with $v$ and not $u$. In this case, the order of the dangling edges in each set of the partition is not relevant due to the symmetry of the Petersen graph.

A \emph{Halin graph} is a plane graph consisting of a planar representation of a tree without degree 2 vertices, and a circuit on the set of its leaves (see \cite{Hal64}).

Let $H$ be a cubic Halin graph consisting of the tree $T$ and the circuit $K$. A \emph{treelike snark} $G$ is any cubic graph that can be obtained by the following procedure:
\begin{itemize}
\item for every leaf $x$ of $T$, we add two new vertices, say $x_1$ and $x_2$, and the edges $xx_1$ and $xx_2$; and
\item for every edge $xy$ of $K$, with $x$ being the predecessor of $y$ with respect to the clockwise orientation of $K$, the edge $xy$ is replaced with a frumious 4--pole, and the first and second left dangling edges of this 4--pole are joined to $x_1$ and $x_2$, respectively, whilst the first and second right dangling edges are joined to $y_1$ and $y_2$, respectively.
\end{itemize}

Let $G$ be a treelike snark as defined above, and let the tree and the circuit defining $G$ be $T$ and $K$, respectively. Let $A$ be a frumious 4--pole of $G$ replacing an edge of $K$. We say that $A$ is of Type $ij$ with respect to a perfect matching $M$ of $G$ if $M$ intersects the left and right dangling edges of $A$ exactly $i$ and $j$ times, respectively, for some $i,j\in\{0,1,2\}$ with $i+j\equiv 0\pmod 2$. We shall denote this by Type$(A_{M})=ij$.

In what follows we shall refer to the first and second left dangling edges of the 4--pole $A$ as $^{-}\kern-0.4em A$ and $_{-}\kern-0.05emA$, respectively. The first and the second right dangling edges are similarly denoted by $A\kern-0.1em^{-}$ and $A_{-}$, respectively (see Figure \ref{FigureConsecutiveStrongPoles}).

Two leaves $x$ and $y$ of $T$ are called \emph{consecutive} if they are adjacent in the circuit $K$, and we shall say that the frumious 4--pole of $G$ replacing the edge $xy$ of $K$ is \emph{in between} the two leaves $x$ and $y$.
Moreover, two consecutive leaves are said to be \emph{near} if they have distance two in $T$, i.e. they have a common neighbour in $T$ (see Figure \ref{FigureNear+Cases1+2}). We remark that $T$ always has two near leaves. Similarly, two $4$--poles $A$ and $B$ are called \emph{consecutive} if there exist three consecutive leaves $x,y,z$ (i.e. $x$ and $y$ are consecutive and $y$ and $z$ are consecutive) such that $A$ is in between $x$ and $y$, and $B$ is in between $y$ and $z$ (see Figure \ref{FigureConsecutiveStrongPoles}). Again, we say that the leaf $y$ is \emph{in between} the $4$--poles $A$ and $B$.

\begin{figure}[h]
      \centering
      \includegraphics[width=0.5\textwidth]{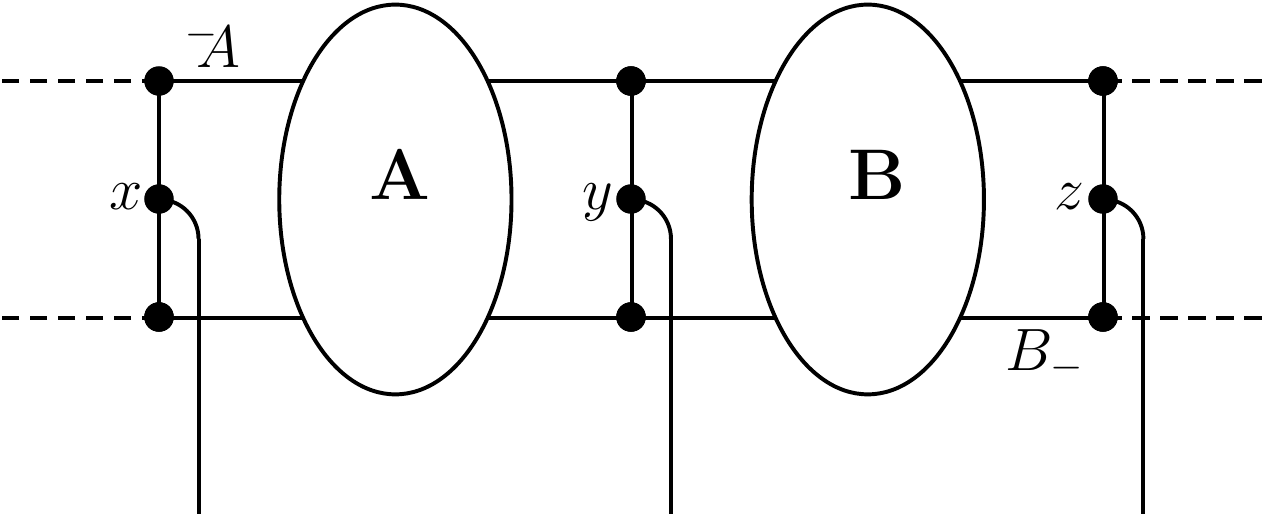}
      \caption{Consecutive leaves and 4--poles}
      \label{FigureConsecutiveStrongPoles}
\end{figure}

\begin{theorem}\label{TheoremTreelikeSnarks}
Every treelike snark is frumious.
\end{theorem}

\begin{proof}
Let $G$ be a treelike snark. We need to prove that $l_{M}(G)=+\infty$ for every perfect matching $M$ of $G$. Suppose, for contradiction, that $G$ is a counterexample having the tree $T$ defining $G$ of minimum order. This means that $G+tM$ is Class I, for some perfect matching $M$ of $G$ and some positive integer $t$. Let the $t+3$ colours of $G+tM$ be the perfect matchings $F_{1}, \ldots, F_{t+3}$.
It is already proved in \cite{WindmillEspMazz} that $scc(G)>\sfrac{4}{3}\cdot|E(G)|$ if $T$ has exactly one vertex of degree $3$, and so, $l_M(G)=+\infty$ by Theorem \ref{thm:sccvslM}. Therefore, we can assume that $T$ has at least two vertices having degree $3$. \\

\textbf{Claim I:} If a 4--pole of $G$ is of Type 00 with respect to $M$, then there must exist exactly one perfect matching from the list of colours $F_{1},\ldots, F_{t+3}$ which intersects both the left (similarly right) dangling edges. 

\emph{Proof of Claim I.} Since the 4--pole of $G$ is of Type 00 with respect to $M$, every dangling edge is contained in exactly one of the colours from the above list. Moreover, since the 4--pole is frumious, exactly one of these colours must intersect both left dangling edges and exactly one of these colours must intersect both right dangling edges (such a colour could be the same for the left and right dangling edges), otherwise one could construct a $(t+3)$--edge-colouring of the frumious snark obtained from the 4--pole (see Definition \ref{def:frumious4pole}), a contradiction. In this case, all the other colours from the list ($t+1$ or $t+2$ of them) do not intersect the four dangling edges.\\

\textbf{Claim II:} If a 4--pole of $G$ is of Type 11 with respect to $M$, then, there must be $t$ perfect matchings from the list of colours $F_{1},\ldots, F_{t+3}$, such that each of them intersects exactly one left dangling edge and exactly one right dangling edge simultaneously. Moreover, there must also exist exactly one perfect matching from the same list which intersects both the left (similarly right) dangling edges of the 4--pole.

\emph{Proof of Claim II.} Since the 4--pole is frumious, at least one colour, say $F_i$, must intersect both the left (right) dangling edges of the 4--pole, by the same argument used in the proof of Claim I, and once again, such colour could be the same for the left and right dangling edges.
Since one of the left (right) dangling edges does not belong to $M$ and belongs to $F_i$, every other colour cannot intersect this left (right) dangling edge. Hence, every perfect matching from the list of colours $F_{1},\ldots, F_{t+3}$ different from $F_i$ intersects the other left (right) dangling edge at most once. More precisely, $t$ of the colours different from $F_{i}$ intersect the left and right dangling edges belonging to $M$ exactly once.\\


\textbf{Claim III:} $G$ cannot contain two consecutive 4--poles which are respectively of Type $00$ and Type $11$ with respect to $M$.

\emph{Proof of Claim III.} If two consecutive 4--poles of $G$ are respectively of Type $00$ and $11$ with respect to $M$, then, the edge of $T$, say $e$, incident to the leaf in between these two 4--poles does not belong to $M$.
On the other hand, by Claim I, there exists a colour $F_i$ which contains both the right (left) dangling edges of the 4-pole of Type 00, and so it contains the edge $e$, as well. By Claim II, there exists a colour $F_j$  which contains both the left (right) dangling edges of the 4--pole of Type 11, and so it contains the edge $e$ too. We note that $j\neq i$, otherwise, $F_i$ contains two pairs of incident edges. Consequently, the edge $e$ belongs to two different colours and so it must belong to $M$, a contradiction.\\

\textbf{Claim IV:} If the unique edge of $T$ incident to a leaf $x$ is not in $M$, then $x$ is in between a 4--pole of Type $11$ and a 4--pole of Type $02$ or, by symmetry, a 4--pole of Type $20$ and a 4--pole of Type $11$, with respect to $M$. 

\emph{Proof of Claim IV.} If the unique edge of $T$ incident to a leaf $x$ is not in $M$, then one of the other two edges incident to $x$ belongs to $M$. Hence, $x$ is in between two 4--poles, one of Type 11 and the other one either of Type 00 or of Type 02 (by symmetry Type 20), with respect to $M$. The first possibility is already excluded by Claim III and so the claim follows.\\

\textbf{Claim V:} If two consecutive leaves are incident with edges in $M$ not belonging to $T$, then the 4--pole in between them is of Type $11$ with respect to $M$.

\emph{Proof of Claim V.} Let $x$ and $y$ be the two consecutive leaves and let $A,B,C$ be the three consecutive 4--poles such that $x$ is in between $A$ and $B$ and $y$ is in between $B$ and $C$. By Claim IV, either $A$ is of Type 20 and $B$ of Type 11  or $A$ is of Type 11 and $B$ of Type 02, with respect to $M$. The latter case is excluded by considering the pair $B$ and $C$ of consecutive 4--poles and Claim IV again. The claim follows. \\

\textbf{Claim VI:} $G$ cannot have three consecutive leaves which are incident with edges in $M$ not belonging to $T$. 

\emph{Proof of Claim VI.} Assume there exist such three consecutive leaves, say $x,y,z$. By Claim V, the two 4--poles in between $x$ and $y$ and in betweeen $y$ and $z$ are both of Type $11$ with respect to $M$. This implies that the edge of $T$ incident with $y$ is in $M$, a contradiction.\\

Next, consider two near leaves of $T$, say $x$ and $y$, as in Figure \ref{FigureNear+Cases1+2}. Let $e$ and $f$ be the two edges of $T$ incident with $x$ and $y$, respectively. Moreover, let $g$ be the edge of $T$ adjacent to $e$ and $f$.
The perfect matching $M$ can intersect $e,f,g$ in two different ways:\\

\textbf{Case 1:} the edge $g$ does not belong to $M$ and exactly one of $e$ and $f$ belongs to $M$, say $e$ without loss of generality, or

\textbf{Case 2:} the edge $g$ belongs to $M$.

\begin{figure}[h]
      \centering
      \includegraphics[width=0.45\textwidth]{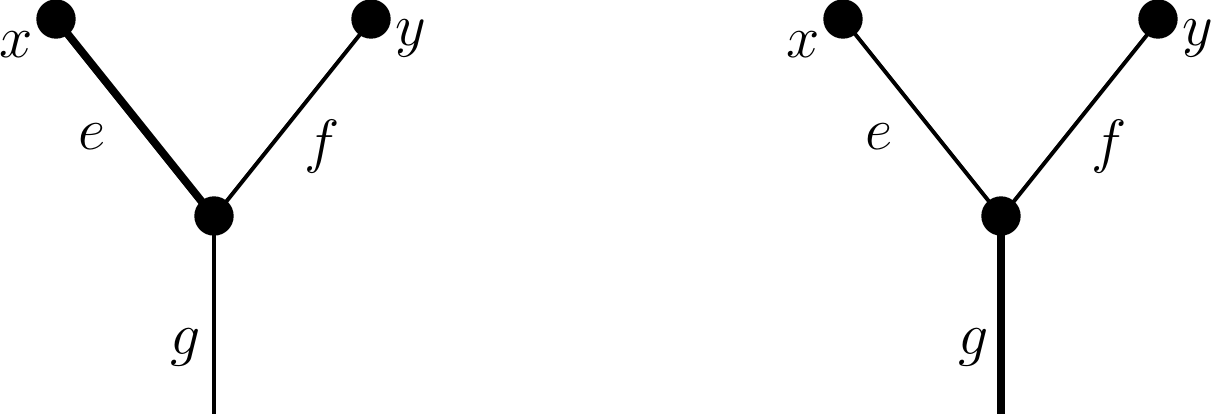}
      \caption{Near leaves $x$ and $y$ in Case 1 and Case 2 of Theorem \ref{TheoremTreelikeSnarks}}
      \label{FigureNear+Cases1+2}
\end{figure}

Consider the three consecutive $4$--poles $A,B,C$ such that $x$ is in between $A$ and $B$, and $y$ is in between $B$ and $C$, as in Figure \ref{FigureCase1b}. The list of proven claims give some strong restrictions and information on the possible types of these 4--poles with respect to $M$. We briefly discuss them according to Type $(B_{M})$, ending with a summary in Table \ref{TableTypesCases1and2}:
\begin{itemize}
 \item Type $(B_{M})$ cannot be equal to 22 or 02, since $f \notin M$ both in Case 1 and Case 2;
 \item If Type $(B_{M})=00$, then Type $(C_{M})=11$ since $f \notin M$, a  contradiction by Claim III;
 \item If Type $(B_{M})=11$, then Type $(C_{M})=02$ since $f \notin M$ and by Claim III. Moreover, if $e \in M$, then Type $(A_{M})=11$ (Case 1b)), otherwise, if $g \in M$, then Type $(A_{M})=20$ (Case 2));
 \item If Type $(B_{M})=20$, then Type $(C_{M})=11$ since $f \notin M$, and Type $(A_{M})$ can be either $00$ (Case 1a)) or $20$ (Case 1c)).
\end{itemize}

\begin{table}[h]
\centering
 \begin{tabular}{ r   c  c  c }
 \toprule
Case  & Type$(A_{M})$ & Type$(B_{M})$ & Type$(C_{M})$ \\
\midrule
 1a) & 00 & 20 & 11 \\
 1b) & 11 & 11 & 02 \\
 1c) & 20 & 20 & 11\\
 2) & 20 & 11 & 02 \\
\bottomrule
\end{tabular}
\caption{}
\label{TableTypesCases1and2}
\end{table}
We prove a further last claim.

\textbf{Claim VII:}
Let $D$ and $D'$ be two consecutive 4--poles of $G$ which are respectively of Type $20$ and $11$ (or by symmetry $11$ and $02$) with respect to $M$. There cannot exist a colour $F_{j}$ such that Type$(D_{F_j})=$ Type$(D'_{F_j})=11$, and there cannot exist a colour $F_{l}$ such that Type$(D_{F_l})=02$ or $22$ (or by symmetry $20$ or $22$).

\emph{Proof of Claim VII.} Consider two consecutive 4--poles $D$ and $D'$ which are respectively of Type $20$ and $11$ (or $11$ and $02$) with respect to $M$. Clearly, the edge $h$ belonging to $T$ and incident with the leaf in between them is not in $M$. Since we are assuming that $G+tM$ is Class I, by Claim II there must exist a colour, say $F_{i}$, such that Type$(D'_{F_i})$, or from now on simply Type$(D'_{i})$, is equal to $20$ or $22$. Clearly, $h\in F_{i}$. This means that there cannot exist a colour $F_{j}$ such that Type$(D_{j})=$ Type$(D'_{j})=11$, as otherwise, $h$ would be covered more than it should be. For the same reasons, there cannot exist a colour $F_{l}$ such that Type$(D_{l})=02$ or $22$ (by symmetry $20$ or $22$).\\

Now, we shall use all previous claims to show that in all the four remaining cases we obtain a contradiction.

\textbf{Case 1a):} Type$(A_{M})=00$, Type$(B_{M})=20$, Type$(C_{M})=11$.\\
Since $A$ is frumious, by Claim I there exists a colour from $F_{1},\ldots,F_{t+3}$, say $F_{1}$, such that Type$(A_1)=\alpha2$, where $\alpha$ is either equal to 0 or 2. As the edges $e$ and $f$ are adjacent, the 4--poles $B$ and $C$ must be intersected by $F_{1}$ as in Table \ref{TableCase1a}. In order to cover the right dangling edges of $B$, there must also exist two colours, say $F_{2}$ and $F_{3}$, such that Type$(B_2)=$ Type$(B_3)=11$, since by Claim VII, Type$(B_i)$ cannot be equal to $02$, for any $i\in\{1,\ldots,t+3\}$. Hence, by Claim VII, $F_{2}$ and $F_{3}$ intersect the 4--poles $A,B,C$ as shown in Table \ref{TableCase1a}, where $\beta,\gamma,\delta,\epsilon\in\{0,2\}$. In any case, this means that the edge $g$ of $T$ is covered twice by $F_{2}$ and $F_{3}$ in $\cup_{i=1}^{t+3}F_{i}$, a contradiction, since $g\not\in M$.

\begin{table}[h]
\centering
\begin{tabular}{ c   c  c  c }
\toprule
$i$ & Type$(A_i)$&  Type$(B_i)$&  Type$(C_i)$  \\ \midrule
1 & $\alpha2$ &  00&  11 \\
2 & $\beta0$ &  11&  $0\delta$ \\
3 & $\gamma0$ &  11&  $0\epsilon$\\
\bottomrule
\end{tabular}
\caption{Case 1a)}
\label{TableCase1a}
\end{table}
\textbf{Case 1b):} Type$(A_{M})=11$, Type$(B_{M})=11$, Type$(C_{M})=02$. \\
There must be a colour from $F_{1}, \ldots, F_{t+3}$, here denoted by $a$, intersecting both the right dangling edges of $A$. In what follows, if $Z$ is a set of colours, we denote by $\overline{Z}$ the set of all colours not in $Z$.
Without loss of generality, assume $A\kern-0.1em^{-}\in M$. Let $b$ and $c$ be the colours of the two edges of $G$ adjacent to $A\kern-0.1em^{-}$. Thus, the colours of $A\kern-0.1em^{-}$ are $\overline{\{b,c\}}$, for simplicity denoted by $\overline{bc}$. Without loss of generality, let the colour of $^{-}\kern-0.2emB$ be $c$. This implies that $c$ also intersects $_{-}\kern-0.05emB$, since by Claim II there must be one colour which intersects both the left dangling edges of $B$, and consequently the colours of $_{-}\kern-0.05emB$ are $\overline{ad}$, for some $d\in \overline{abc}$. Moreover, the edge $e$ has colours $\overline{bd}$, as can be seen in Figure \ref{FigureCase1b}.

\begin{figure}[h]
      \centering
      \includegraphics[width=.5\textwidth]{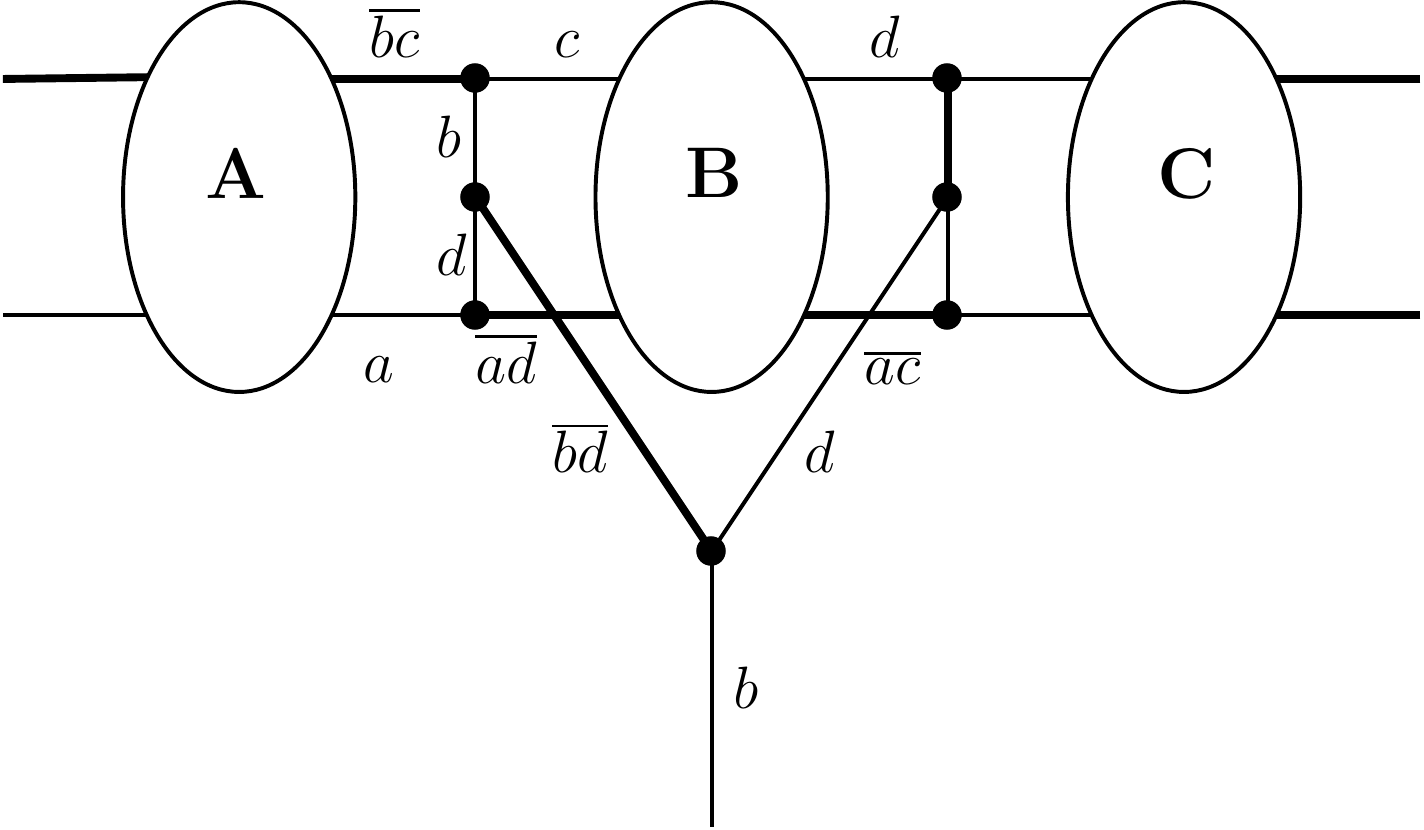}
      \caption{Case 1b)}
      \label{FigureCase1b}
\end{figure}

Once again, by Claim II, there is a colour which intersects both the right dangling edges of $B$. Clearly, this cannot belong to $\overline{bd}$, for otherwise, $f$ would be coloured by a colour already used for $e$. Since $b$ intersects exactly one left dangling edge of $B$, the right dangling edges of $B$ must be intersected by $d$. Without loss of generality, we can assume that $B_{-}\in M$, and so by the above reasoning, the set of colours of $B_{-}$ is $\overline{ac}$ (see Figure \ref{FigureCase1b}).
At this point, we have two possible cases of how we can colour $^{-}\kern-0.1emC$ and $_{-}\kern-0.05emC$: we either have $^{-}\kern-0.1emC$ and $_{-}\kern-0.05emC$ intersected by $a$ and $c$, respectively, or the other way round, as can be seen in the two figures in Figure \ref{FigureCases1b}:

\begin{figure}[H]
\begin{subfigure}{.5\textwidth}
  \centering
  \includegraphics[width=.96\textwidth]{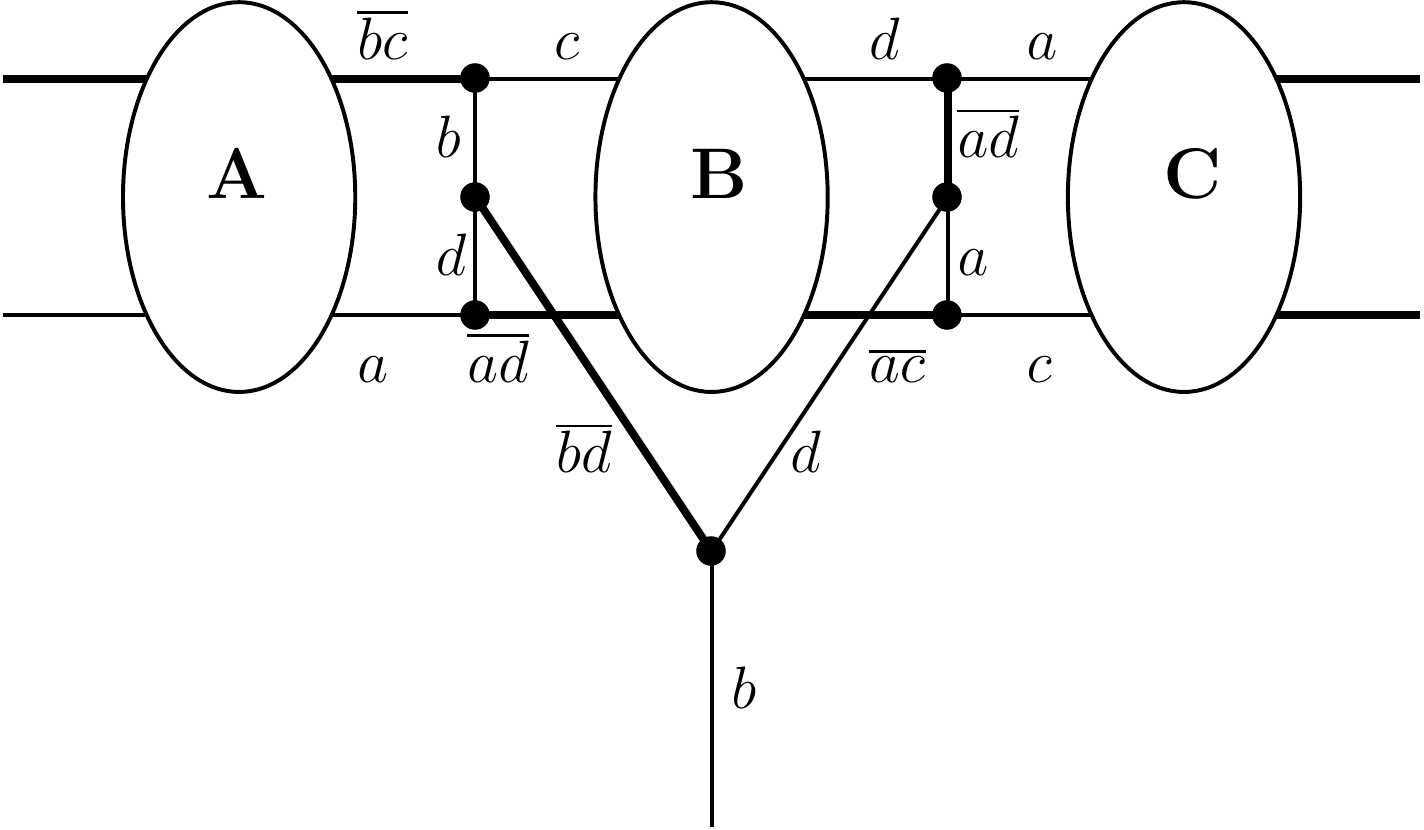}
\label{FigureCase1b1}
\end{subfigure}
\begin{subfigure}{.5\textwidth}
  \centering
  \includegraphics[width=.96\textwidth]{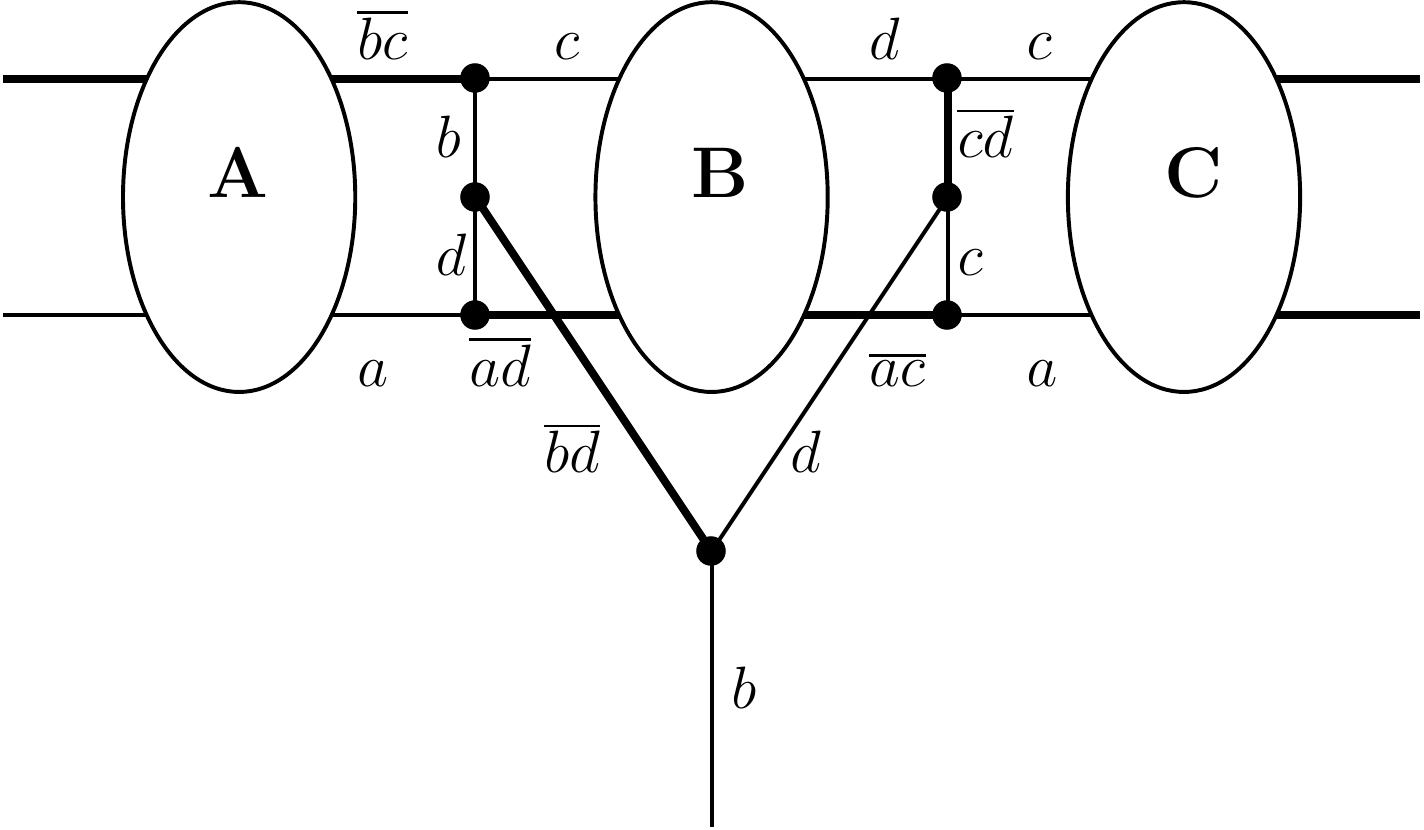}
\label{FigureCase1b2}
\end{subfigure}
\caption{The possible colours of $^{-}\kern-0.1emC$ and $_{-}\kern-0.05emC$ in Case 1b)}
\label{FigureCases1b}
\end{figure}

Next, we reduce $G$ to a smaller treelike snark following the procedure presented in Figure \ref{FigureConstructingSmallerTreelikeSnark}. Since $T$ has at least two vertices of degree $3$, the resulting graph $G'$ is indeed a treelike snark. Let $T'$ be the tree defining $G'$.

\begin{figure}[h]
\centering
      \includegraphics[width=.9\textwidth]{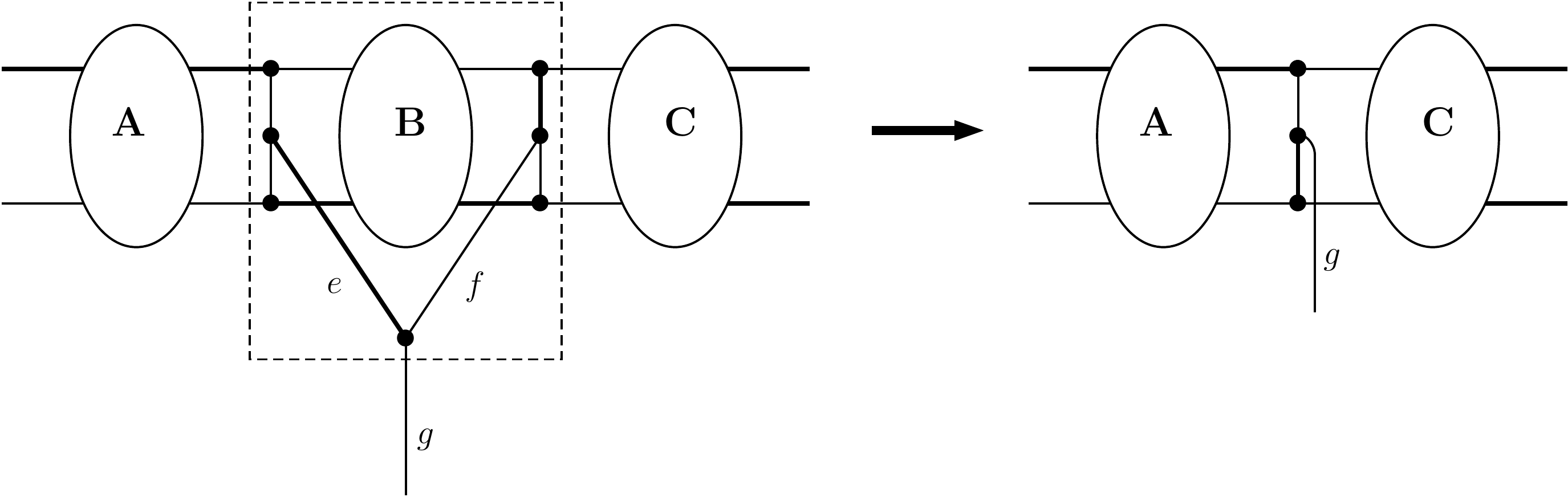}
\caption{Constructing a smaller treelike snark in Case 1b)}
\label{FigureConstructingSmallerTreelikeSnark}
\end{figure}

Let $M'$ be the perfect matching of $G'$ induced by $M$. Without loss of generality, assume that the colours of $^{-}\kern-0.1emC$ and $_{-}\kern-0.05emC$ in $G$ are $a$ and $c$, respectively, and assign to the edges of $G'$ (which correspond to edges of $G$) the same colours they had originally. We note that this procedure does not colour all the edges of $G'$. In fact, the two edges not belonging to $T'$ which are incident to the leaf between the two 4--poles $A$ and $C$ in $G'$, do not correspond to any edges of $G$, and so they are left uncoloured. Moreover, the edges of $G'$ are not properly coloured, as depicted in Figure \ref{FigureCase1b1KempeChain}.

\begin{figure}[h]
      \centering
      \includegraphics[width=0.7\textwidth]{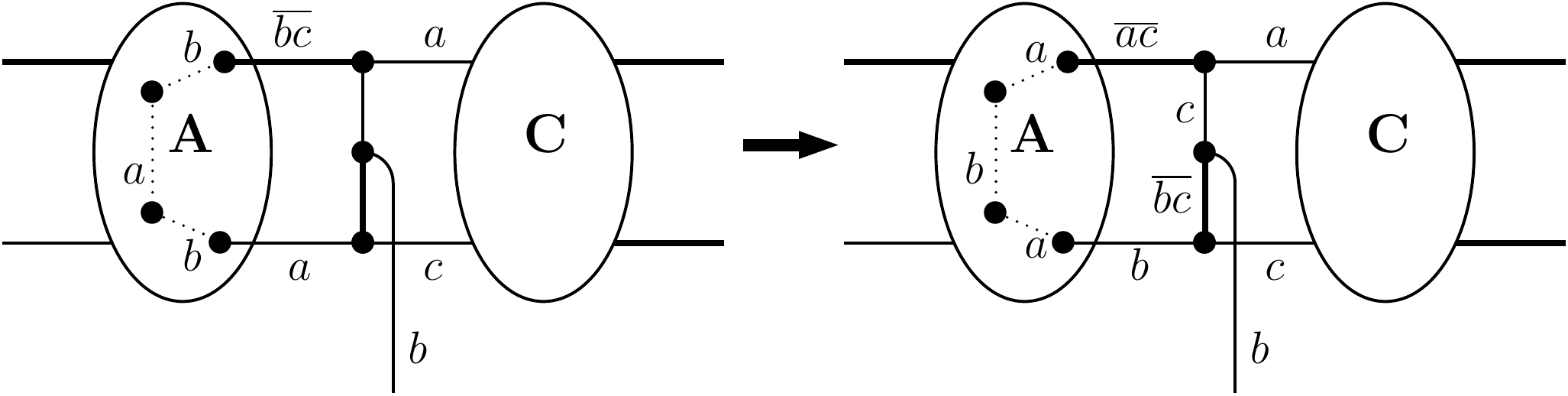}
      \caption{Applying a Kempe chain argument in Case 1b)}
      \label{FigureCase1b1KempeChain}
\end{figure}

We claim that the $(a,b)$--Kempe chain in the 4--pole $A$ starting at $A_{-}$ must end at $A\kern-0.1em^{-}$. For, suppose it does not contain the latter dangling edge. Let $M_{A}$ be the perfect matching induced by $M$ in the 4--pole $A$. Switching the colours $a$ and $b$ along this chain shall result in a $(t+3)$--edge-colouring of the pole $A+tM_{A}$ in which no colour intersects the two right dangling edges of $A$ simultaneously, contradicting Claim II. Consequently, the $(a,b)$--Kempe chain in $G'$ starting from $A_{-}$ must end at $A\kern-0.1em^{-}$. By switching the colours $a$ and $b$ along this chain and extending the colouring to a $(t+3)$--edge-colouring of $G'+tM'$ as in Figure \ref{FigureCase1b1KempeChain}, we obtain a contradiction due to the minimality of $T$.

\textbf{Case 1c):} Type$(A_{M})=20$, Type$(B_{M})=20$, Type$(C_{M})=11$.\\
This case is solved in a similar way as in Case 1b), and so this case cannot occur as well. Figure \ref{FigureCase1c} shows the four different ways how a $(t+3)$--edge-colouring of $G+tM$ looks like in this part of $G$.

\begin{figure}[h]
\begin{subfigure}{.5\textwidth}
  \centering
  \includegraphics[width=.96\textwidth]{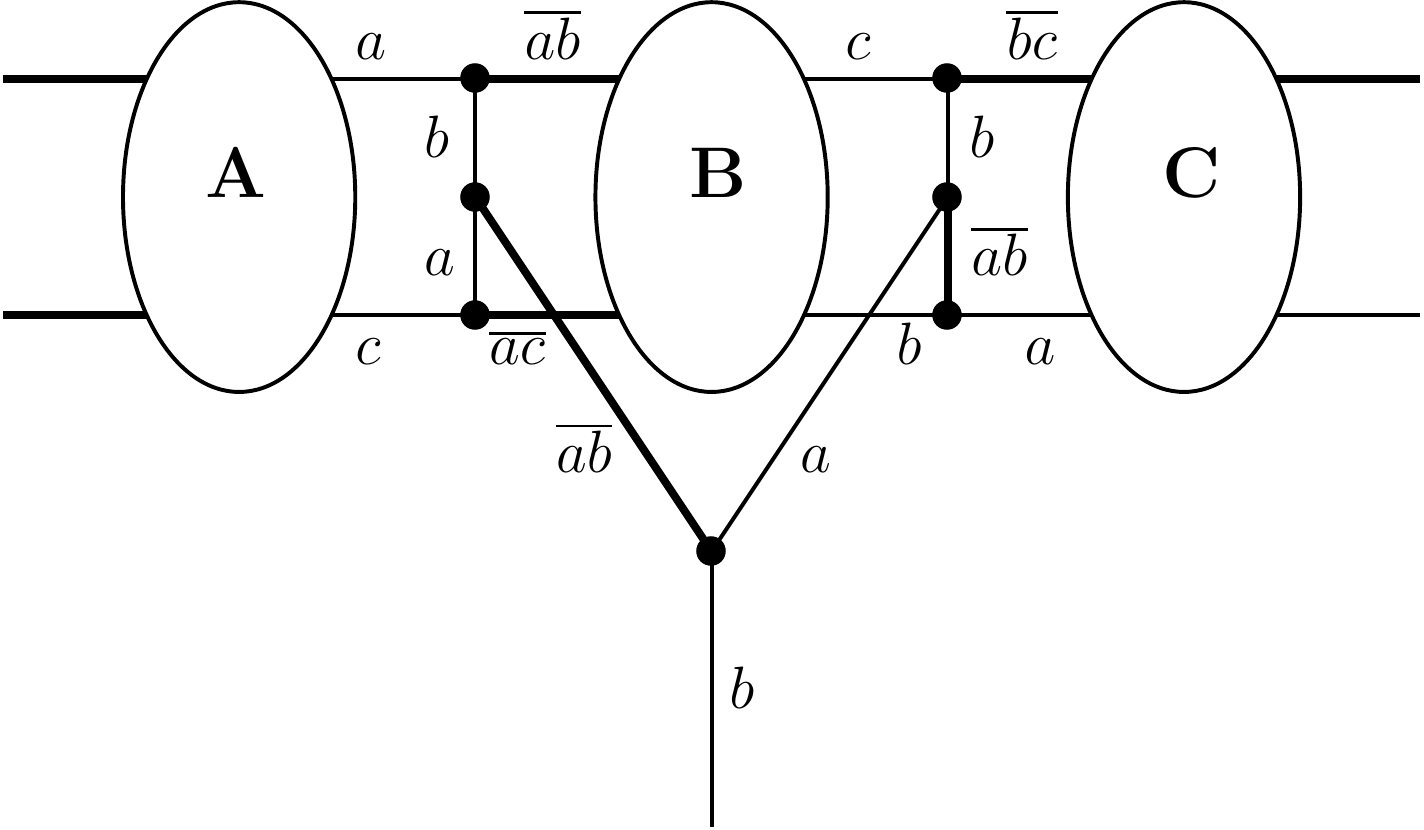}
\end{subfigure}%
\begin{subfigure}{.5\textwidth}
  \centering
  \includegraphics[width=.96\textwidth]{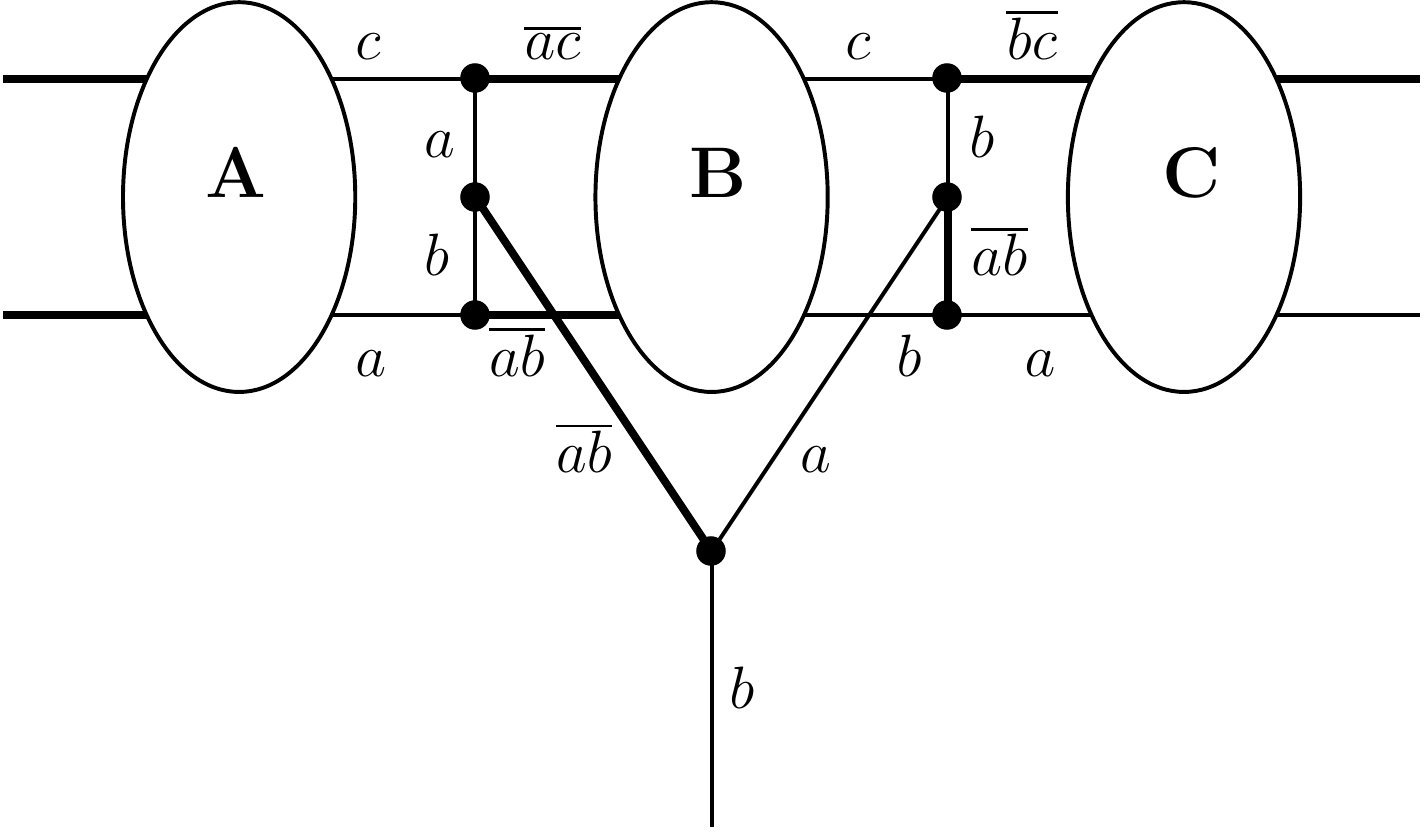}
\end{subfigure}
\begin{subfigure}{.5\textwidth}
  \centering
  \includegraphics[width=.96\textwidth]{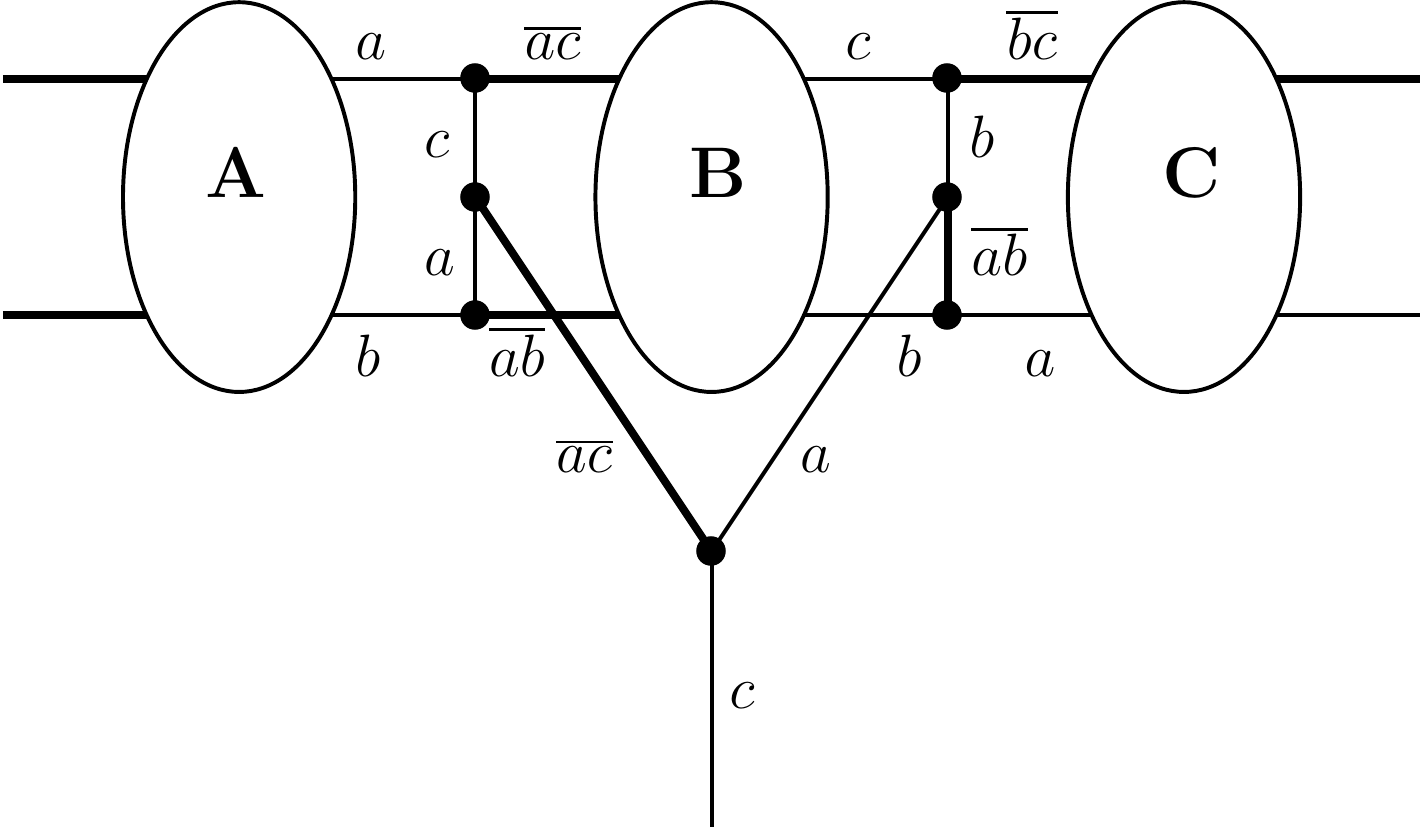}
\end{subfigure}%
\begin{subfigure}{.5\textwidth}
  \centering
  \includegraphics[width=.96\textwidth]{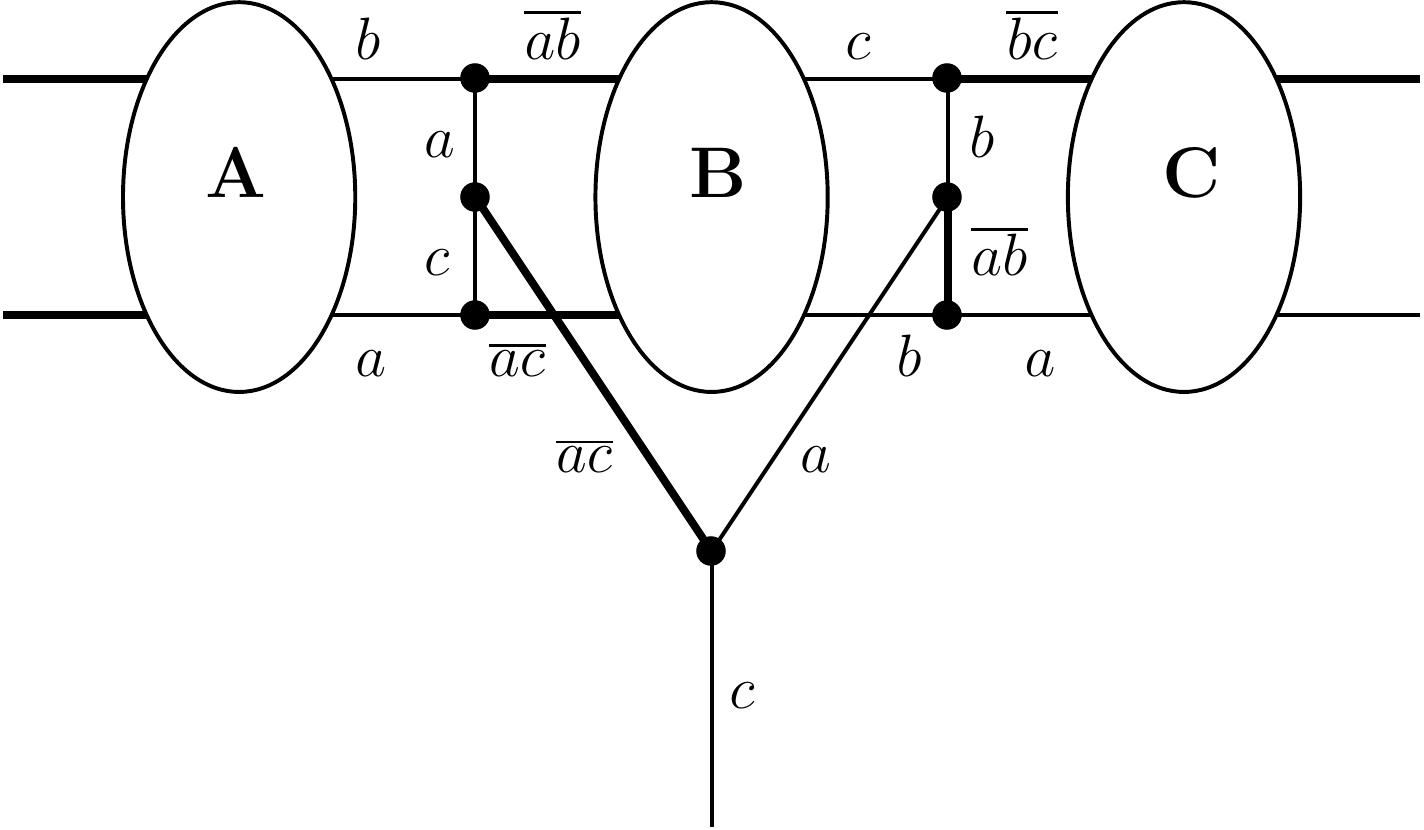}
\end{subfigure}

\caption{Case 1c)}
\label{FigureCase1c}
\end{figure}
\textbf{Case 2:} Type$(A_{M})=20$, Type$(B_{M})=11$, Type$(C_{M})=02$.\\
Since all other cases are not possible, all pairs of near leaves of $G$ are in between three 4--poles of these types (with respect to $M$). We show that in such a case, there exist three consecutive leaves all incident with edges in $M$ not belonging to $T$, a contradiction by Claim VI.
In fact, if $T$ has only two  vertices of degree 3, it has exactly two pairs of near leaves, with all the four edges incident with the leaves not belonging to $M$. Thus, we have three consecutive leaves with the required property, contradicting Claim VI.

Therefore, $T$ must have more than two vertices of degree 3. Remove all pairs of near leaves from $T$, and let the resulting tree, which still has all vertices of degree 1 and 3, be $T'$.
In general, if $x$ is a leaf of $T'$ which was not a leaf in $T$, then the edge in $T'$ incident to $x$ belongs to $M$.
Consider a pair of near leaves of $T'$. At least one of them was not a leaf in $T$, as otherwise the pair would have been deleted in the process of obtaining $T'$. If these two near leaves were not originally leaves in $T$, then the edges in $T'$ incident with them both belong to $M$, a contradiction. Hence, one leaf of the pair must also be a leaf in $T$, whilst the other leaf of the pair was a common neighbour to a pair of removed near leaves of $T$. Consequently, $G$ contains three consecutive leaves such that the edges of $T$ incident with them do not belong to $M$, contradicting Claim VI once again.
Hence, $l_{M}(G)=+\infty$ for every perfect matching $M$ of $G$.
\end{proof}

We complete this section with the following corollary which simply follows by Theorem \ref{TheoremTreelikeSnarks} and Theorem \ref{thm:sccvslM}.

\begin{corollary}
For every treelike snark $G$,  $scc(G)>\sfrac{4}{3}\cdot|E(G)|$. 
\end{corollary}

\section{Final remarks and related problems}

In the following table, we summarise all the parameters discussed along the paper and we recall two of the main conjectures proposed above. In particular, the table highlights the special role of the class ${\cal S}_{\geq 5}$ with regards to all the considered problems.

\begin{center}
\begin{tabular}{ c  c  c  c  c }
\toprule
$\chi'(G)$ & $\chi'_e(G)$ & $scc(G)$ & $l(G)$ & $l_M(G)$  \\
\midrule
 $3$ & $3$ & $\sfrac{4}{3}\cdot|E(G)|$ & $0$ & $(\forall M) \, 0$\\
\midrule
\multirow{3}*{$4$} & $4$       & $\sfrac{4}{3}\cdot|E(G)|$ & $1$ & $(\exists M) \, 1$\\[1ex]
                & \multirow{2}*{$\geq 5$ } & \multirow{2}*{$>\sfrac{4}{3}\cdot|E(G)|$} & \multirow{2}*{ $>1$ }&  \multirow{2}*{$(\forall M) >1$} \\[2ex]

         & & ( Conj.\ref{Conjectureperfect matchingIndex4SCC} ) & & ( Conj.\ref{conj:G+tM}: $(\forall M)=+\infty$ )\\

\bottomrule
\end{tabular}

\end{center}

Let us remark that the problem of establishing the existence of a perfect matching $M$ for which $l_M(G)$ is finite is equivalent to establishing the existence of a $2$--factor (indeed the complement of $M$ in $G$) which can be extendend to a cycle double cover of $G$. This problem was already considered for some classes of snarks (see for instance \cite{HagglundOstenhof}). We remark that Conjecture \ref{Conjectureperfect matchingIndex4SCC} can be equivalently stated in such terms as follows:

\begin{conjecture}For every bridgeless cubic graph $G$, $\chi_{e}'(G)> 4$ if and only if every cycle double cover of $G$ does not contain a $2$--factor of $G$.
\end{conjecture}

The other main conjecture in this paper is Conjecture \ref{conj:G+tM}: we were not able to find an example of a snark $G$ for which $l_M(G)$ is finite but larger than $1$, for every perfect matching $M$ of $G$. Moreover, in Section \ref{Section scc} we showed that the conjecture holds for a large family of snarks having perfect matching index 5.
A first natural step in an attempt to understand better Conjecture \ref{conj:G+tM} is trying to solve the following problem:

\begin{problem}
\label{prob:2Mclass1} Characterise the class of bridgeless cubic graphs $G$ for which there exists a perfect matching $M$, such that $G+2M$ is Class I.
\end{problem}

Clearly, if Conjecture \ref{conj:G+tM} holds we would have a complete answer to the previous problem, and the graphs $G$ answering Problem \ref{prob:2Mclass1} would be those bridgeless cubic graphs having perfect matching index at most $4$.\\

Finally, as one can notice, along the paper we mainly focus our attention on $1$--factors of $G$. A very similar problem for $2$--factors of a snark $G$ was communicated personally to us by Eckhard Steffen.
\begin{problem}[Steffen, personal communication]\label{ProblemLeastNo2Factors}
Let $G$ be a bridgeless cubic graph. What is the smallest number of 2--factors that need to be added to $G$, such that the resulting graph is Class I?
\end{problem}

The classical Berge-Fulkerson Conjecture is equivalent to saying that the answer for Problem \ref{ProblemLeastNo2Factors} is at most $1$, with the answer being $0$ if $G$ is already a Class I graph.
Here, we would like to propose a possible approach for the study of this problem.

For a bridgeless cubic graph $G$, let $sp_2(G)$ be the set of all non-negative integers $t$ such that $G$ contains $t$ 2--factors whose addition to $G$ results into a Class I graph, and let $sp(G)$ be the set of all non-negative integers $t$ such that $tG$ is Class I, where $tG$ represents $G+(t-1)E(G)$. These two parameters are related in the following way:

\begin{proposition}\label{prop:2factorSpecter}
For any bridgeless cubic graph $G$ and any integer $t\geq 0$, $t\in sp_2(G)$ if and only if $(t+1)\in sp(G)$.
\end{proposition}

\begin{proof} Assume that $t\in sp_2(G)$. Then, there are $t$ 2--factors $\overline{F_1},...,\overline{F_t}$ of $G$ such that $G+\overline{F_1}+...+\overline{F_t}$ is Class I, i.e. $(2t+3)$--edge-colourable. Hence, there are $2t+3$ perfect matchings $J_1,...,J_{2t+3}$ that partition the edge set of the graph $G+\overline{F_1}+...+\overline{F_t}$, and consequently
\[\Vec{1}=\chi^{J_1}+...+\chi^{J_{2t+3}}-\chi^{\overline{F_1}}-...-\chi^{\overline{F_t}}.\]
Let $F_{i}$ be the perfect matching $E(G)-\overline{F_{i}}$, for every $i\in\{1,\ldots, t\}$. By noting that for every $i$, $\Vec{1}=\chi^{F_{i}}+\chi^{\overline{F_{i}}}$, we have
\[\Vec{1}=\chi^{J_1}+...+\chi^{J_{2t+3}}-(\Vec{1}-\chi^{{F_1}})-...-(\Vec{1}-\chi^{{F_t}}),\]
which implies that
\[(t+1)\Vec{1}=\chi^{J_1}+...+\chi^{J_{2t+3}}+\chi^{{F_1}}+...+\chi^{{F_t}}.\]
The latter means $(t+1)G$ is Class I, i.e. $(3t+3)$--edge-colourable. The converse can be similarly proved using the same arguments.
\end{proof}

\end{document}